\documentclass[10pt, reqno]{amsart}
\usepackage{a4}
\usepackage{amssymb}
\usepackage{array}
\usepackage{booktabs}
\usepackage{multirow}
\usepackage{hhline}

\newtheorem{thm}{Theorem}

\newtheorem{cor}[thm]{Corollary}
\newtheorem{lem}{Lemma}
\newtheorem{prop}[thm]{Proposition}
\newtheorem{rem}[thm]{Remark}

\numberwithin{thm}{section}
\numberwithin{equation}{section}

\newcommand{\norm}[1]{\left\Vert#1\right\Vert}
\newcommand{\abs}[1]{\left\vert#1\right\vert}

\newcommand{\eps}{\varepsilon}

\newcommand{\A}{\mathcal{A}}
\newcommand{\As}{\text{\sffamily{A}}}
\newcommand{\B}{\mathcal{B}}
\newcommand{\Bs}{\text{\sffamily{B}}}
\newcommand{\Cs}{\text{\sffamily{C}}}
\newcommand{\ds}{\text{\sffamily{d}}}
\newcommand{\E}{\mathbb{E}}
\newcommand{\Es}{\text{\sffamily{E}}}

\newcommand{\F}{\mathcal{F}}

\newcommand{\Hi}{\mathcal{H}}
\newcommand{\K}{\mathcal{K}}

\newcommand{\ku}{\text{\upshape{k}}}

\newcommand{\M}{\mathcal{M}}
\newcommand{\n}{\mathbb{N}}
\newcommand{\N}{\mathcal{N}}

\newcommand{\tphi}{\widetilde{\varphi}}
\newcommand{\U}{\mathcal{U}}

\newcommand{\z}{\mathbb{Z}}

\newcommand{\Om}{\Omega}

\begin{document}
\title[Maurey]
{A Maurey type result for operator spaces}
\author{Marius Junge \and Hun Hee Lee}

\address{Marius Junge : Department of Mathematics 
University of Illinois at Urbana-Champaign 
1409 W. Green Street, 
Urbana, Illinois, USA 61801-2975}
\address{Hun Hee Lee : Department of Pure Mathematics, Faculty of Mathematics, University of Waterloo,
200 University Avenue West, Waterloo, Ontario, Canada N2L 3G1}

\begin{abstract}
The little Grothendieck theorem for Banach spaces says that every bounded linear operator between $C(K)$ and $\ell_2$ is 2-summing.
However, it is shown in \cite{J05} that the operator space analogue fails. Not every cb-map $v : \K \rightarrow OH$ is completely 2-summing.
In this paper, we show an operator space analogue of Maurey's theorem : Every cb-map $v : \K \rightarrow OH$ is $(q,cb)$-summing for any $q>2$
and hence admits a factorization $\|v(x)\| \leq  c(q)  \|v\|_{cb} \|axb\|_q$ with $a,b$ in the unit ball of the Schatten class $S_{2q}$.
\end{abstract}

\maketitle

\section{Introduction}

The theory of operator spaces investigates subspace of $C^*$-algebras with their inherited matricial structure.
Many concepts from Banach space theory can be formulated in the setting of so-called ``quantized Banach spaces''.
In particular Grothendieck's fundamental work on tensor norms leads to many interesting new problems in the context of operator algebras
and operator spaces. Let us mention in particular Shlyakhtenko and Pisier's version of Grothendieck's theorem for operator spaces,
Haagerup and Musat's the very recent completion of Grothendieck's theorem for $C^*$-algebras and the results in \cite{J05, P04, PS02}.
A fundamental object in the theory of operator spaces is Pisier's operator space $OH$,
the only operator space completely isometric to its anti-dual. Using their version of Grothendieck's theorem for operator spaces,
Pisier-Shlyakhtenko obtained a characterization of completely bounded maps $$u : A \rightarrow OH$$ for every $C^*$-algebra $A$:
Indeed $u$ is completely bounded if and only if there exists a state $\phi$ and a constant $C>0$ such that
$$ \norm{u(x)} \leq C [\phi(x^*x)\phi(xx^*)]^{\frac{1}{4}}.$$
This characterization should be considered as analogue of the little Grothendieck's theorem in the theory of Banach spaces.
It is shown in \cite{J05} that a straight forward translation
 \begin{equation}\label{lg} \pi_2^{o}(u)\leq C \norm{u}_{cb}?
 \end{equation}
does not hold in general, and not even uniformly for finite dimensional $C^*$-algebra's $A$.
We will define the completely $2$-summing norm $\pi_2^o$ below.

In this paper we will approach from a different angle.  Let us first recall the classical Banach space theory.
Let $X$ be a Banach space with cotype $q$, i.e.
 \[ \Big( \sum_k \|Tx_k\|_X^q \Big)^{\frac{1}{q}}\leq c_q(X) \E \|\sum_k \eps_k x_k\|_X \]
holds for all finite families $x_1,...,x_n\in X$, where $(\eps_k)_{k \geq 1}$ is the classical Rademacher sequence
and $\E$ is the corresponding expectation. Maurey (\cite{Mau}) showed  that for a Banach space with cotype $q$ we have
 \begin{equation}\label{Ma}
 L(C(K),X) = \Pi_p(C(K),X)
 \end{equation}
holds for every $p>q$ and every space of continuous functions $C(K)$. Even for $X=L_q([0,1])$ and $q>2$ the result is not true for $p=q$.
Let us recall that a map $T:X\rightarrow Y$ is $p$-summing if
 \[ \Big( \sum_{k=1}^n  \|Tv(e_k)\|_Y^p\Big)^{\frac{1}{p}}
 \leq C \|v:\ell_{p'}^n\rightarrow X\|.\]
Then $\pi_p(T)=\inf C$, where the infimum is taken over all constants satisfying the inequality above for arbitrary $u$.
In the setting of operator spaces we can easily adapt this notation and say that
$T:E\rightarrow F$ is $(p,cb)$-summing if
 \[ \Big( \sum_{k=1}^n  \|Tv(e_k)\|_F^p\Big)^{\frac{1}{p}}
 \leq C \|v:\ell_{p'}^n\rightarrow E\|_{cb}.\]
holds for some constant $C$. As above we define $\pi_{p,cb}(T)=\inf C$ and $\Pi_{p,cb}(E,F)$ as the space of $(p,cb)$-summing maps.
Very little is known about the right concept of cotype $q$, although some attempts have been made in the literature (see \cite{Par03, JP05} and
\cite{L-typecotype, L-weak}). Clearly, we should expect that $OH$ has cotype $2$.
In this sense our main result is an operator space version of Maurey's theorem:

\begin{thm}\label{thm-main} Let $2<q<\infty$. Then
 \[ CB(B(H),OH)\subseteq \Pi_{q,cb}(B(H),OH).\]
\end{thm}

The factorization theory for $(q,cb)$-maps is very satisfactory (see \cite{J-Hab, P98}).
In the finite dimensional setting, the result reads as follows: Let $$u:M_m\rightarrow OH$$ be a completely bounded map.
Then there are positive elements $a, b$ with $$\norm{a}_{S_{2q}} \,,\, \norm{b}_{S_{2q}} \leq 1$$ such that 
\[ \|u(x)\| \leq  c(q)  \|u\|_{cb} \|axb\|_q.\]
Note that the statement fails for $q=2$ and indeed we have $c(q)\leq c_0 (\frac{q}{q-2})^{\frac{1}{2}}$ for some constant $c_0$.

The definition of $(p,cb)$-summing maps lies in between Banach space and operator space theory.
In operator space theory a map is called completely $p$-summing if
 \[ \norm{[ Tu(e_{ij})]}_{S_p^n(F)} \leq C \|u:S_{p'}^n\rightarrow E\|_{cb}.\]
Then $\pi_p^o(T)=\inf C$, and $\Pi_p^o(E,F)$ is the space of completely $p$ summing maps between operator spaces $E$ and $F$.
Let us recall that for a matrix $x=[x_{ij}]$ with values in $E$ the norm in $S_p^n(E)$ is defined as
 \[ \|x \|_{S_p^n(E)} = \inf_{x_{ij}=\sum_{kl} a_{ik}y_{kl}b_{kj}} \|a\|_{2p}  \|[y_{ij}]\|_{M_n(E)} \|b\|_{2p},\]
where $a=[a_{ij}]$ and $b=[b_{ij}]$.
Note that every operator space carries a natural family of matrix norms $M_n(E)$. We refer to \cite{P98} for more details and
properties of the vector-valued noncommutative $L_p$ spaces. It is well-known that completely $p$-summing maps are completely bounded.
Therefore it is tempting to formulate the following strengthening of our result.

\vspace{0.5cm} \noindent {\bf Problem:} Let $2<q<\infty$. It is true that
 \begin{equation}\label{conj}  CB(B(H),OH) = \Pi_q^o(B(H),OH)?
 \end{equation}

\noindent Our approach to Theorem \ref{thm-main} uses duality. We first show that the conclusion is equivalent to 
\begin{equation}\label{dual-formulation}
\ell_{p}(\ell_2)  = \Pi_1^o(OH, \ell_p).
\end{equation}
Let us note that also \eqref{conj} is equivalent to \[S_{p}(OH) = \Pi_{1}^o(OH, S_p). \]
Here $\frac{1}{p}+\frac{1}{q}=1$ is the conjugate index. Following the general theory of completely $1$-summing maps
we can realize the space $\Pi_1^o(OH, \ell_p)$ as a subspace of a noncommutative $L_1$ space.
Here we invoke the results and methods from the recent paper \cite{JP-LpLq} which shows that $\ell_{p}$ is completely isomorphic to
subspace of a noncommutative $L_1$ space with respect to a von Neumann algebra with $QWEP$.
Recall that a $C^*$-algebra $A$ has $WEP$ (Lance's weak expectation property) if the inclusion map $i_A : A \hookrightarrow A^{**}$ 
factors completely positively and completely contractively through $B(H)$ for some Hilbert space $H$.
A $C^*$-algebra $B$ has $QWEP$ if there is a $WEP$ $C^*$-algebra $A$ and two sided ideal $I\subseteq A$ such that $B \cong A/I$.

Based on recent results of Xu on embedding results using tools from real interpolation theory
(see \cite{X-Rp-Embedding}) and Pisier's concrete embedding of $OH$ using generalized free gaussian variables (see \cite{PS02} and \cite{P04}),
we can identify a rather concrete embedding of $\Pi_{1}^o(S_p,OH)$. We are then able to show that at least for the identity
$id:\ell_p^n\rightarrow \ell_2^n=OH_n$ we have
 \begin{equation}\label{funds}
\pi_1^o(id:\ell_2^n\rightarrow\ell_{p}^n ) \sim \Big(\frac{q}{q-2}\Big)^{\frac{1}{2}} n^{\frac{1}{p}}.
 \end{equation}
Unfortunately, calculating this norms turns out to be rather delicate and requires a detailed case by case analysis in a 8-term quotient space.
Using further properties of our concrete realization, we can then find an intermediate vector-valued Orlicz norm
estimating the completely $1$-summing from above and the norm in $\ell_{p'}(\ell_2)$ from below.
Testing the Orlicz norm on the sum of the unit vectors we obtain the full result from \ref{funds}.
Using the theory of tensor norms in operator space, we can formulate the following application.

\begin{cor} Let $1< p<2$. Then
 \[ \Pi_p^o(OH,\ell_{p}) = \Pi_1^o(OH,\ell_{p}) \]
with equivalent norms.
\end{cor}

The paper is organized as follows. We collect some preliminaries in section \ref{sec-prelim}.
In section \ref{sec-dualProb} we present the dual formulation \eqref{dual-formulation}. This requires us several embedding results into 
a noncommutative $L_1$ space, which will be given in the following section. In section \ref{sec-embeddings} we combine
the ideas of Junge, Xu, Pisier and Junge $\&$ Parcet of embedding $OH$ and $S_p$ ($1<p<2$) into noncommutative $L_1$ spaces.
In section \ref{sec-ChangeOfDensity} we use the information from the previous section to find a concrete embedding of $\Pi_{1}^o(S_p,OH)$.
In section \ref{sec-identity} we do the calculation for the identity, which is crucial to our conclusion.
In the last section we apply the ``Orlicz space argument" by Junge and Xu to explain that the result for the identity is enough to show our main result.

\section{Preliminaries and Notations}\label{sec-prelim}
We assume that the reader is familiar with standard concepts in operator algebra (\cite{Ta79, Ta03}) and operator space theory (\cite{ER00, P03}).

For two operator spaces $E_0$ and $E_1$ we denote their $\ell_p$-direct sum by $E_0 \oplus_p E_1$ for $1\leq p \leq \infty$ (\cite{P03}). 
If $(E_0, E_1)$ is a pair of operator spaces which is a compatible pair in the Banach space sense,
then $E_0 +_p E_1$ refers to the quotient operator space of $E_1 \oplus_p E_2$ by the subspace $\{(x_0, x_1) : x_0 +x_1 =0\}$.
Similarly $E_0 \cap_p E_1$ refers to the diagonal subspace of $E_0 \oplus_p E_1$. Note that $E_0 \oplus_p E_1$'s are all completely isomorphic for 
$1\leq p \leq \infty$ with a universal constant and so are $E_0 +_p E_1$'s and $E_0 \cap_p E_1$'s.
When $p=1$ we simply write $E_0 +_1 E_1$ as $E_0 + E_1$.
We will prefer $E_0 +_2 E_1$ and $E_0 \cap_2 E_1$ in section \ref{sec-embeddings} to be more precise in constant,
while we prefer $E_0 + E_1$ in the following sections since we have 
$$(E_0 + E_1)\widehat{\otimes}\,(F_0 + F_1) \cong (E_0 \widehat{\otimes}\, F_0) + (E_1 \widehat{\otimes}\, F_1)$$
completely isometrically, where $\widehat{\otimes}$ is the projective tensor product of operator spaces.

For a Hilbert space $H$ we denote the column, the row and the operator Hilbert space on $H$ by $H^c$, $H^r$ and $H^{oh}$, respectively.
For $1\leq p \leq \infty$ and $n\in \n$ we denote $R^n_p = [R_n, C_n]_{\frac{1}{p}}$, where $[\cdot, \cdot]_{\frac{1}{p}}$
implies complex interpolation in the operator space sense (\cite{P96}).

We will frequently use noncommutative $L_1$ spaces in this paper.
For a $\sigma$-finite von Neumann algebra $\A$ with a distinguished normal faithful state $\phi$ with density $D$
the noncommutative $L_1$-space in the sense of Haagerup is denoted by $L_1(\A)\, (= L_1(\A ,\phi))$.
There is a natural operator space structure on $L_1(\A)$ as the predual of $\A$. 

Vector valued $L_1$-spaces can be defined for $R^n_1$, $C^n_1$ and $OH_n$ as follows.
	\begin{align*}
	L_1(\A ; R^n_1) & := \Big\{\sum^n_{i=1}x_i \otimes e_{1i}\Big\} \subseteq L_1(\A \overline{\otimes}M_n),\\
	L_1(\A ; C^n_1) & := \Big\{\sum^n_{i=1}x_i \otimes e_{i1}\Big\} \subseteq L_1(\A \overline{\otimes}M_n)
	\end{align*}
and
$$L_1(\A; OH_n) := \Big[L_1(\A ; R^n_1),\,  L_1(\A ; C^n_1)\Big]_{\frac{1}{2}}.$$

Let $\As$ be a sub-von Neumann algebra of $\A$ and $\Es : \A \rightarrow \As$ a normal faithful conditional expectation satisfying
$$\phi = \phi|_\As \circ \Es.$$ Then, the space $L^r_1(\A, \Es)$ and $L^c_1(\A, \Es)$ (\cite{J02}) are defined by
the completions of $D\A$ and $\A D$ under the norms
$$\norm{Dx}_{L^r_1(\A, \Es)} = \norm{(D\Es(xx^*)D)^{\frac{1}{2}}}_{L_1(\As)}\,\, \text{and}\,\,
\norm{xD}_{L^c_1(\A, \Es)} = \norm{(D\Es(x^*x)D)^{\frac{1}{2}}}_{L_1(\As)},$$ respectively.

Since $L_2(\A)$ is a Hilbert space we can consider $L^{r_1}_2(\A)$ and $L^{c_1}_2(\A)$ endowed with operator space structures
in the sense of $R_1 = C$ and $C_1 = R$, then their operator space structure can be described as follows.
Let $\text{tr}_\A$ the unique tracial functional on $L_1(\A)$ satisfying $$\phi(a) = \text{tr}_\A (aD)$$ for all $a\in \A$. Then we have
\begin{align*}
\norm{(I_{S^m_1}\otimes D^{\frac{1}{2}})a}_{S^m_1(L^{r_1}_2(\A))} & =
\norm{(I_{S^m_1} \otimes \text{tr}_{\A})((I_{S^m_1}\otimes D^{\frac{1}{2}})aa^*(D^{\frac{1}{2}} \otimes I_{S^m_1}))^{\frac{1}{2}}}_{S^m_1}\\
& = \norm{(I_{M_m}\otimes \phi)(aa^*)^{\frac{1}{2}}}_{S^m_1}
\end{align*}
and 
\begin{align*}
\norm{b(D^{\frac{1}{2}}\otimes I_{S^m_1})}_{S^m_1(L^{c_1}_2(\A))}
& = \norm{(I_{S^m_1} \otimes \text{tr}_{\A})((I_{S^m_1} \otimes D^{\frac{1}{2}})b^*b(I_{S^m_1} \otimes D^{\frac{1}{2}}))^{\frac{1}{2}}}_{S^m_1}\\
& = \norm{(I_{M_m}\otimes \phi)(b^*b)^{\frac{1}{2}}}_{S^m_1}
\end{align*}
for $a,b \in S^m_1 \otimes \A$ and $m\in \n$.

We use the symbol $a\lesssim b$ if there is a $C>0$ such that $a \leq C b$ and $a\sim b$ if $a\lesssim b$ and
$b\lesssim a$.

\section{The dual problem}\label{sec-dualProb}

We present a dual formulation of the original problem, which enables us to do concrete calculations.
For a linear map $v: E\rightarrow F$ between operator spaces we consider
$\Gamma_{\infty}$-norm and $\gamma_{\infty}$-norm of $v$ defined by
$$\Gamma_{\infty}(v) = \inf\norm{\alpha}_{cb}\norm{\beta}_{cb},$$
where the infimum is taken over all Hilbert space $H$ and the factorization
$$i_F v : E \stackrel{\alpha}{\rightarrow} B(H) \stackrel{\beta}{\rightarrow} F^{**},\,\,
\text{where $i_F$ is the inclusion $F \hookrightarrow F^{**}$}$$  and
$$\gamma_{\infty}(v) = \inf\norm{\alpha}_{cb}\norm{\beta}_{cb},$$
where the infimum is taken over all $m\in \mathbb{N}$ and the factorization
$$v : E \stackrel{\alpha}{\rightarrow} M_m \stackrel{\beta}{\rightarrow} F.$$
See section 4 of \cite{J05} or \cite{EJR00} for the details.

\begin{thm}\label{thm-dualprob}
Let $1< p < 2$ and $\frac{1}{p} + \frac{1}{p'} = 1$. Then, the following conditions are equivalent.
\begin{itemize}
\item[(1)] For any Hilbert space $H$ we have $$CB(B(H), OH) \subseteq \Pi_{p', cb}(B(H), OH).$$
\item[(2)] There is a constant $C>0$ such that
$$\pi^o_1(T_x : OH \rightarrow \ell_p) \leq C \norm{x}_{\ell_p(OH)}$$ for all $x \in \ell_p(OH)$ and
$T_x : OH \rightarrow \ell_{p}$, the linear map naturally associated to $x$.
\item[(3)] $\Pi^o_p(OH, \ell_p) \subseteq \Pi^o_1(OH, \ell_p)$.
\end{itemize}
\end{thm}
\begin{proof}
(1) $\Rightarrow$ (2)

By a standard density argument it is enough to consider $n$-dimensional case, $n\in \n$, $\ell^n_p(OH_n)$
instead of $\ell_p(OH)$. Then, since $\Gamma_{\infty} = \gamma_{\infty}$ for linear maps
between finite dimensional spaces (see \cite{EJR00}) and $\gamma_{\infty}$ is the trace dual of $\pi^o_1$,
(2) is equivalent to
	\begin{eqnarray}\label{equi-(2)}
	\norm{y}_{\ell^n_{p'}(OH_n)} \leq C \cdot \Gamma_{\infty}(T^y : \ell^n_p \rightarrow OH_n)
	\end{eqnarray}
for all $y \in \ell^n_{p'}(OH_n)$ and
$T^y : \ell^n_p \rightarrow OH_n$, the linear map naturally associated to $y$.
Now for any $\epsilon > 0$ we have a factorization
$T^y : \ell^n_p \stackrel{\alpha}{\rightarrow} B(H) \stackrel{\beta}{\rightarrow} OH_n$ with
	$$\norm{\alpha}_{cb}\norm{\beta}_{cb} \leq (1+\epsilon)\Gamma_{\infty}(T^y).$$
Then, for $y = \sum^n_{i=1}e_i \otimes y_i \in \ell^n_{p'}(OH_n)$ we have
\begin{align*}
\begin{split}
\norm{y}_{\ell^n_{p'}(OH_n)} & = \Big(\sum^n_{i=1}\norm{y_i}^{p'}_{OH_n}\Big)^{\frac{1}{p'}} =
\Big(\sum^n_{i=1}\norm{T^y e_i}^{p'}_{OH_n}\Big)^{\frac{1}{p'}}\\ & \leq
\pi_{p', cb}(T^y) \norm{\sum^n_{i=1}e_i \otimes e_i}_{\ell^n_{p'} \otimes_{\min}\ell^n_p}\\
& = \pi_{p', cb}(T^y) \norm{\ell^n_p \rightarrow \ell^n_p,\,\, e_i \mapsto e_i}_{cb}\\ & = \pi_{p', cb}(T^y)
 \leq \pi_{p', cb}(\beta)\norm{\alpha}_{cb}\\ & \leq C\norm{\beta}_{cb}\norm{\alpha}_{cb} \leq C(1+\epsilon)\Gamma_{\infty}(T^y)
\end{split}
\end{align*}
for some constant $C > 0$ coming from the inclusion (1).
\vspace{0.4cm}

(2) $\Rightarrow$ (1)

With the same reason as above it is enough to consider $OH_n$ instead of $OH$.
Let $u : B(H) \rightarrow OH_n$. Then for any $(x_i)^m_{i=1} \subseteq B(H)$ and
$v : \ell^m_p \rightarrow B(H), \,\, e_i \mapsto x_i$ we have by (\ref{equi-(2)})
\begin{align*}
\begin{split}
\Big(\sum^n_{i=1}\norm{ux_i}^{p'}_{OH_n}\Big)^{\frac{1}{p'}} & =
\Big(\sum^n_{i=1}\norm{uv e_i}^{p'}_{OH_n}\Big)^{\frac{1}{p'}} \leq C \cdot\Gamma_{\infty}(uv)\\
& \leq C\norm{u}_{cb}\norm{v}_{cb} = C\norm{u}_{cb}\norm{\sum^m_{i=1}e_i\otimes x_i}_{\ell^m_{p'}\otimes_{\min} B(H)},
\end{split}
\end{align*}
which implies $\pi_{p', cb}(u) \leq C\norm{u}_{cb}$.
\vspace{0.4cm}

(2) $\Longleftrightarrow$ (3)

Again, we are enough to consider finite dimensional cases.
Note that there is a completely isomorphic embedding $OH_n \stackrel{i}{\hookrightarrow} L_p(\M)$
for a von Neumann algebra with $QWEP$. This noncommutative $L_p$ space is understood in the sense of Haagerup.
Then, by Corollary 10 of \cite{Y-p-summing} we have
\begin{align*}
\begin{split}
\pi^o_p(T_x : OH_n \rightarrow \ell^n_p) & \sim \pi^o_p(T_x \circ i^* : i(OH_n)^* \rightarrow \ell^n_p)\\
& = \norm{I_{\ell^n_p} \otimes i (x)}_{\ell^n_p(L_p(\M))} \sim \norm{x}_{\ell^n_p(OH_n)}
\end{split}
\end{align*}
for any $x \in \ell^n_p(OH_n)$.

\end{proof}

\begin{rem}\label{rem-dualprob}{\rm
By a similar argument as the above theorem we can show that for $1< p < 2$ and
$\frac{1}{p} + \frac{1}{p'} = 1$ the followings are equivalent.
	\begin{itemize}
	\item[$(1')$] For any Hilbert space $H$ we have $$CB(B(H), OH) \subseteq \Pi^o_{p'}(B(H), OH).$$
	\item[$(2')$] There is a constant $C>0$ such that
		$$\pi^o_1(T_x : OH \rightarrow S_p) \leq C \norm{x}_{S_p(OH)}$$ for all $x \in S_p(OH)$ and
	$T_x : OH \rightarrow S_{p}$, the linear map naturally associated to $x$.
	\item[$(3')$] $\Pi^o_p(OH, S_p) \subseteq \Pi^o_1(OH, S_p)$.
	\end{itemize}
At the time of this writing we could not answer this question.

If we look at the condition $(3')$, then (3) of Theorem \ref{thm-dualprob} is a particular case the above question,
which we are dealing with diagonals. Thus, it is natural to consider columns and rows as the next candidate of particular cases.
That is to say we are interested in the following question.
	\begin{itemize}
	\item[$(3'')$] $\Pi^o_p(OH, C_p) \subseteq \Pi^o_1(OH, C_p)$. (resp. $\Pi^o_p(OH, R_p) \subseteq \Pi^o_1(OH, R_p),$)
	\end{itemize}
which is true and can be explained in a similar way yet the calculation is much simpler.
}\end{rem}

Now we focus on the $n$-dimensional ($n\in \n$) case of (2) of Theorem \ref{thm-dualprob}.
The right-hand side term $\norm{x}_{\ell^n_p(OH_n)}$ is easy to describe,
so the point is to describe the left-hand side term $\pi^o_1(T_x : OH_n \rightarrow \ell^n_p)$ in a concrete way.

Suppose there are embeddings $$OH \stackrel{i}\hookrightarrow E \subseteq L_1(\M)\;\text{and}\;
\ell_p \stackrel{j}\hookrightarrow F \subseteq L_1(\N)$$ for some von Neumann algebras $\M$ and $\N$ with $QWEP$
and cb-projections $$P : L_1(\M) \rightarrow E \; \text{and}\; Q : L_1(\N) \rightarrow F$$ with
$P|_E = I_E$ and $Q|_F = I_F$, then by Lemma 4.4 and 4.5 in \cite{J05} we have
\begin{align*}
\begin{split}
\pi^o_1(T_x : OH_n \rightarrow \ell^n_p) & \sim \pi^o_1(j \circ T_x \circ i^* : i(OH_n)^* \rightarrow j(\ell^n_p))\\
& = \norm{i \otimes j (x)}_{L_1(\M) \widehat{\otimes} L_1(\N)}
\end{split}
\end{align*}
for all $x \in OH_n \otimes \ell^n_p$.
Thus, it would be the first task to find such embeddings with $E$ and $F$ are concrete spaces, which will be considered in the following section.

\section{Embeddings of various spaces into the noncommutative $L_1$ space with respect to a von Neumann algebra with $QWEP$}\label{sec-embeddings}

\subsection{Some aspects of real interpolation approach.}\label{subsec-real-interpolation}

When we want to embed $OH$ into the predual of a von Neumann algebra it is very important to observe that
it is completely isomorphic to a subspace of quotient of $R\oplus C$ (\cite{J05, P04, X-Rp-Embedding}).
Similarly, the embedding of $S_p$ ($1< p <2$) (\cite{JP-LpLq}) starts with the observation that $C_p$ and $R_p$ are
completely isomorphic to a subspace of quotient of $R\oplus OH$ and $C\oplus OH$, respectively.
In this section we review the real interpolation approach by Xu (\cite{X-Real-Embedding, X-Rp-Embedding}) to the above observations.

Let $1<p<\infty$, $\theta = \frac{1}{p}$ and $\alpha \in \mathbb{R}$. For a Banach space $X$ we denote $X$-valued
$L_2(\mathbb{R}^{+}, t^{2\alpha}\frac{dt}{t})$ space by $L_2(t^\alpha ; X)$. Now we let 
$$K_{\theta} = L^c_2(t^{-\theta}; \ell_2) +_2 L^r_2(t^{1-\theta}; \ell_2)
\,\, \text{and}\,\, J_{\theta} = L^c_2(t^{-\theta}; \ell_2) \cap_2 L^r_2(t^{1-\theta}; \ell_2).$$

Let $C_{\theta; K}$ be the subspace of $K_{\theta}$ consisting of constant functions and
$C_{\theta; J}$ be the quotient space of $J_{\theta}$ by the subspace of mean zero functions.
If we look at the Banach space level then $C_{\theta; K}$ and $C_{\theta; J}$ are nothing but the interpolation of $\ell_2$ with itself,
so that we clearly recover $\ell_2$ regardless of $\theta$. However by posing column and row Hilbert space structure in the above way we get
a completely isomorphic copy of $C_p$, which now depends on $\theta = \frac{1}{p}$.
Note that $(C_{\theta; J})^* = C_{1- \theta; K}$ completely isometrically.

\begin{prop}\label{prop-Rp-Embedding}
Let $1<p<\infty$ and $\theta = \frac{1}{p}$. Then, $C_p$ and $C_{\theta; K}$ are completely isomorphic allowing constant depending only on
$\theta$. More precisely, we have $$\norm{\sum^n_{i,j=1}x_{ij}\otimes {\bf 1}\otimes e_{ij}}_{M_m(C_{\theta; K})}
\sim \theta^{-\frac{1}{2}}(1-\theta)^{-\frac{1}{2}}\norm{\sum^n_{i,j=1}x_{ij}\otimes e_{ij}}_{M_m(C_p)},$$
where ${\bf 1}$ implies the constant scalar function with value $1$.
\end{prop}
\begin{proof}
See Theorem 3.3 of \cite{X-Rp-Embedding}. Note that the factor of $\theta^{-\frac{1}{2}}(1-\theta)^{-\frac{1}{2}}$
was ignored in the proof, which should have appeared when we were dealing with the interpolation of two $L_p$ spaces
with different measures (see \cite{BL76}).
\end{proof}

For $1< p <2$ we can consider two variations of the above interpolation. Now we pose row and operator (resp. column and operator)
Hilbert space structure as follows, so that we get $C_p$ (resp. $R_p$).

For $0< \theta < 1$ we let
	$$K_{c, \theta} = L^r_2(t^{-\theta}; \ell_2) +_2 L^{oh}_2(t^{1-\theta}; \ell_2),\,\,
	K_{r, \theta} = L^c_2(t^{-\theta}; \ell_2) +_2 L^{oh}_2(t^{1-\theta}; \ell_2),$$
	$$\,\,J_{c, \theta} = L^c_2(t^{-\theta}; \ell_2) \cap_2 L^{oh}_2(t^{1-\theta}; \ell_2)\;
	\text{and}\,\, J_{r, \theta} = L^r_2(t^{-\theta}; \ell_2) \cap_2 L^{oh}_2(t^{1-\theta}; \ell_2).$$
Let $C_{c, \theta; K}$ (resp. $R_{r, \theta; K}$) be the subspace of $K_{c, \theta}$ (resp. $K_{r, \theta}$)
consisting of constant functions and $R_{c, \theta; J}$ (resp. $C_{r, \theta; J}$)
be the quotient space of $J_{c, \theta}$ (resp. $J_{r, \theta}$) by the subspace of mean zero functions.
Note that $$(R_{c, \theta; J})^* = C_{c, \theta; K}\,\, (\text{resp.}\,\, (C_{r, \theta; J})^* = R_{r, \theta; K})$$
completely isometrically.

\begin{prop}\label{prop-Rp-C-OH-embedding}
Let $1< p < 2$, $\frac{1}{p} + \frac{1}{p'}= 1$ and $\theta = \frac{2}{p'}$. Then,
$C_p$ and $C_{c, \theta; K}$ (resp. $R_p$ and $R_{r, \theta; K}$) are
completely isomorphic allowing constant depending only on $\theta$.
More precisely, we have
	$$\norm{\sum^n_{i,j=1}x_{ij}\otimes {\bf 1}\otimes e_{ij}}_{M_m(C_{c, \theta; K})}
	\sim \theta^{-\frac{1}{2}}(1-\theta)^{-\frac{1}{2}}\norm{\sum^n_{i,j=1}x_{ij}\otimes e_{ij}}_{M_m(C_p)}.$$
The situation for $R_{r, \theta; K}$ is similar.
\end{prop}
\begin{proof}
The following proof is similar to that of Theorem 3.3 of \cite{X-Rp-Embedding}.
Recall that (Theorem 8.4 of \cite{P96}) for $x = (x_k) \in M_m(C_p)$ we have
	$$\norm{x}_{M_m(C_p)} = \sup\Big\{ \Big(\sum_{k\geq 1}\norm{ax_k b}^2_2 \Big)^{\frac{1}{2}} :
	\norm{a}_{S^m_{2p}}, \norm{b}_{S^m_{2p'}} \leq 1,\,\, a,b > 0 \Big\}.$$
For fixed $a$ and $b$ with $\norm{a}_{S^m_{2p}}, \norm{b}_{S^m_{2p'}} \leq 1$ and $a,b > 0$
we consider
	$$A_0 = L_{a^p}\;\, \text{and}\;\, A_1 = L_{a^{\frac{p}{2}}}R_{b^{\frac{p'}{2}}},$$
where $L_{\alpha}$ and $R_{\beta}$ implies left and right multiplications by $\alpha$ and $\beta$,
respectively, on $H = \ell_2(S^m_2)$. Then $A_0$ and $A_1$ are commuting invertible
positive bounded operators on $H$, and $A_i$ induces an equivalent norm $\norm{\cdot}_i$ on $H$ as follows :
	$$\norm{x}_i := \norm{A_i x},\;\, i=0,1.$$
Let $H_i$ be $H$ equipped with $\norm{\cdot}_i$.
Then $(H_0, H_1)$ becomes a compatible couple of Hilbert spaces, which can be identified as
a couple of weighted $L_2$ spaces. Then by real interpolation of $L_2$ spaces with different weight
(see \cite{BL76}) we have
	$$\Big(\sum_{k\geq 1}\norm{ax_k b}^2_2\Big)^{\frac{1}{2}} = \norm{A_0^{1-\theta}A_1^\theta x}
	\sim c^{-1}_{\theta}\norm{x}_{(H_0, H_1)_{2, \theta; K}}$$ for some $c_\theta \sim \theta^{-\frac{1}{2}}(1-\theta)^{-\frac{1}{2}}$.

Now we suppose $\norm{x}_{M_m(C_{c, \theta; K})} < 1$. Then, there are $f \in M_m(L^r_2(t^{-\theta}; \ell_2))$
and $g \in M_m(L^{oh}_2(t^{1-\theta}; \ell_2))$ such that $x = f(t) + g(t)$ for almost all $t \in (0,\infty)$,
	$$\norm{\int^{\infty}_0 \sum_{k\geq 1} f_k(t)f^*_k(t)t^{-2\theta}\frac{dt}{t}}_{M_m} < 1$$
and
	$$\norm{\int^{\infty}_0 \sum_{k\geq 1} g_k(t)\otimes
	\overline{g_k(t)}t^{2(1-\theta)}\frac{dt}{t}}_{M_m \otimes_{\min} \overline{M_m}} < 1.$$
Moreover, we have
	\begin{align*}
	\norm{f}^2_{L_2(t^{-\theta} ;H_0)} & = \int^{\infty}_0 \norm{f(t)}^2_{H_0}t^{-2\theta}\frac{dt}{t}\\
	& = \int^{\infty}_0 \sum_{k\geq 1} \text{tr}_m(a^p f_k(t)f^*_k(t)a^p)t^{-2\theta}\frac{dt}{t}\\
	& = \text{tr}_m\Big(a^{2p}\int^{\infty}_0 \sum_{k\geq 1} f_k(t)f^*_k(t)t^{-2\theta}\frac{dt}{t}\Big)\\
	& \leq \norm{a^{2p}}_1\norm{\int^{\infty}_0 \sum_{k\geq 1} f_k(t)f^*_k(t)t^{-2\theta}\frac{dt}{t}}_{M_m} <1
	\end{align*}
and by $(7.3)'$ of \cite{P03}
\begin{align*}
\begin{split}
\norm{g}^2_{L_2(t^{1-\theta} ;H_1)} & = \int^{\infty}_0 \norm{g(t)}^2_{H_1}t^{2(1-\theta)}\frac{dt}{t}\\
& = \int^{\infty}_0 \sum_{k\geq 1} \text{tr}_m(a^{\frac{p}{2}}g_k(t)b^{p'}g^*_k(t)a^{\frac{p}{2}})t^{2(1-\theta)}\frac{dt}{t}\\
& = \text{tr}_m \Big(\int^{\infty}_0 \sum_{k\geq 1} a^p g_k(t)b^{p'}g^*_k(t)t^{2(1-\theta)}\frac{dt}{t}\Big)\\
& \leq \norm{a^p}_2 \norm{b^{p'}}_2\norm{\int^{\infty}_0 \sum_{k\geq 1} g_k(t)\otimes
\overline{g_k(t)}t^{2(1-\theta)}\frac{dt}{t}}_{M_m \otimes M_m} <1
\end{split}
\end{align*}
Thus, we have $\norm{x}_{(H_0, H_1)_{2, \theta; K}} < \sqrt{2}$,
and consequently
	$$C_{c, \theta; K} \subseteq C_p\;\, \text{with cb-norm}\;\, \leq c^{-1}_{\theta}\sqrt{2}.$$

Using J-method we can similarly show that
	$$R_{c, \theta; J} \subseteq R_p\;\, \text{with cb-norm} \;\, \leq c_{\theta}\sqrt{2}.$$
In this case we need to take $A_0 = R_{b^p}$ and $A_1 = L_{a^{\frac{p'}{2}}}R_{b^{\frac{p}{2}}}$.
Then, by duality we get the desired cb-isomorphism.

The proof for $R_p$ and $R_{r, \theta; K}$ is similar.
\end{proof}

\subsection{The case of $OH$}\label{subsec-L2L1}
In this section we consider the case of $OH$, which was first done by Junge (\cite{J05})
and explained in different forms by Pisier (\cite{P04}) and Xu (\cite{X-Rp-Embedding}).

We will continue to employ the real interpolation approach as in the previous section.
Now we set $\theta = \frac{1}{2}$ and consider a discretization
$K_{\frac{1}{2}, \delta}$ ($1< \delta \leq 2$) of $K_{\frac{1}{2}}$ defined by
\begin{equation}\label{discretization}
K_{\frac{1}{2}, \delta} := \ell^c_2(\delta^{-\frac{k}{2}}; \ell_2) +_2 \ell^r_2(\delta^{\frac{k}{2}}; \ell_2),
\end{equation}
where $\ell_2(\delta^{k\alpha}; \ell_2)$ denotes the weighted $\ell_2(\n)$-valued $\ell_2$ space on $\z$
with respect to the weight $(\delta^{2k\alpha})_{k \in \z}$. Then,
$K_{\frac{1}{2}, \delta}$ is $\delta$-completely isomorphic to $K_{\frac{1}{2}}$.
In order to show that $K_{\frac{1}{2}, \delta}$ can be embedded into the predual of a von Neumann algebra
we need some tools from free probability.

Let $\Hi$ be a Hilbert space with Hilbert space basis $(e_{\pm n})_{n\geq 1}$.
Then we consider the full Fock space $\F(\Hi) =\mathbb{C}\Om \oplus_{n \geq 1}\Hi^{\otimes n}$,
the left creation operator $\ell(e)$ and the left annihilation operator $\ell^*(e)$
on $\F(\Hi)$ associated to $e\in \Hi$. Let $$g_n = \lambda_n^{-\frac{1}{2}}\ell(e_n) +
\lambda_n^{\frac{1}{2}}\ell^*(e_{-n})$$ for some sequence $(\lambda_n)_{n\geq 1}$ of strictly positive real numbers.
These $g_n$'s are called ``generalized circular elements" by Shlyakhtenko (\cite{S97, PS02}), and it is well known that the
von Neumann algebra $\M$ generated by $\{g_n : n\geq 1\}$ has $QWEP$. Moreover, if we let $D_{\Phi}$ be the density of the
vector state $\Phi$ on $\M$ determined by the vacuum vector $\Om$,
then $D_{\Phi}^{\frac{1}{2}}g_n D_{\Phi}^{\frac{1}{2}} \in
L_1(\M)$ and the operator space
$$G_* = \overline{\text{span}}\{D_{\Phi}^{\frac{1}{2}}g_n D_{\Phi}^{\frac{1}{2}} : n\geq 1\} \subseteq L_1(\M)$$
is 2-completely isomorphic to $\ell^c_2(\n, \lambda_n^{\frac{1}{2}}) + \ell^r_2(\n,\lambda_n^{-\frac{1}{2}})$
and is 2-completely complemented in $L_1(\M)$. Note that we have (\cite{S97})
\begin{equation}\label{modular-condition}
D_{\Phi}^{\frac{1}{2}}g_n D_{\Phi}^{\frac{1}{2}} = \lambda_n g_n D_{\Phi} = \lambda^{-1}_n D_{\Phi} g_n.
\end{equation}

Now we go back to our original concern $K_{\frac{1}{2}, \delta}$. If we set
$$\Hi = \ell_2(\z; \ell_2) \oplus_2 \ell_2(\z; \ell_2)$$ with basis
$\{e_k \otimes e_j : k \in \z,\,\, j\in \n\}\cup \{f_k \otimes e_j : k \in \z,\,\, j\in \n\}$
and $$\lambda_{k,j} = \delta^k\,\, \text{for $k\in \z$ and $j\in \n$},$$ then the corresponding
$\M_{\n}\; (= \M^\delta_{\n}) = \{g_{k,j} : k \in \z,\,\, j\in \n \}''$, where
$$g_{k,j}\; (= g^\delta_{k,j}) = \delta^{-\frac{k}{2}}\ell(e_k \otimes e_j) + \delta^{\frac{k}{2}}\ell^*(f_k \otimes f_j),$$
and $$G^{\n}_*\; (= G^{\n}_*(\delta)) = \overline{\text{span}}\{D_{\Phi}^{\frac{1}{2}}g_{k,j} D_{\Phi}^{\frac{1}{2}}
: k \in \z,\,\, j\in \n \}\subseteq L_1(\M_{\n}) $$ is our desired embedding.

More precisely if we set $M(j) = \{g_{k,j} : k \in \z\}''$, then $M_j$'s are all isomorphic and free each other.
Let $\phi_j$ be the restriction of $\Phi$ on $M(j)$, and we set $$(\M_n, \Phi) = *^n_{j=1}(M(j), \phi_j).$$
Note that $\M_{\infty} = \M_{\n}$.
Now we denote 
\begin{equation}\label{notation1}
\text{$M(1)$, $\phi_1$ and $(g_{k,1})_{k\in \z}$ by simply $M$, $\phi$ and $(g_k)_{k\in \z}$,}
\end{equation}
respectively, and let $\rho_j : M \hookrightarrow \M_n = *^n_{j=1}M_j$ be the natural embedding into the $j$-th component.
Then since $$\rho_j(g_k) = g_{k,j}\,\, \text{and}\,\, \rho_j(D_{\phi}^{\frac{1}{2}} x D_{\phi}^{\frac{1}{2}})
= D_{\Phi}^{\frac{1}{2}} \rho_j(x) D_{\Phi}^{\frac{1}{2}}$$ for the density $D_{\phi}$ of $\phi$
and $x\in M$ we have the following with the help of Proposition \ref{prop-Rp-Embedding}. 
This observation is a combination of the ideas in \cite{X-Rp-Embedding} and \cite{P04}.

\begin{prop}\label{prop-OH-Embedding}
Let $n \in \n \cup\{\infty\}$ and $1< \delta \leq 2$. Then $OH_n$ is cb-embedded in a completely complemented subspace 
$$G^n_* = \overline{\text{span}}\{D_{\Phi}^{\frac{1}{2}}g_{k,j} D_{\Phi}^{\frac{1}{2}}
: k \in \z,\,\, 1\leq j \leq n \}\subseteq L_1(\M_{n})$$ with the constants independent of $\delta$ and $n$ by the following embedding.
$$v^\delta_n : OH_n \rightarrow G^n_* \subseteq L_1(*^n_{j=1}M_j), \,\, e_j \mapsto
\rho_j\Big(\sum_{k\in \z} D_{\phi}^{\frac{1}{2}}g_k D_{\phi}^{\frac{1}{2}}\Big) =
\sum_{k\in \z}D_{\Phi}^{\frac{1}{2}}g_{k,j}D_{\Phi}^{\frac{1}{2}}.$$ 
\end{prop}

\subsection{The case of $S_p$ ($1<p<2$)}\label{subsec-SpL1}
In this section we consider the case of $S_p$ ($1<p<2$) following the very recent work of Junge and Parcet (\cite{JP-LpLq}).
The starting point of this embedding is the factorization $$S_p = C_p \otimes_h R_p$$ and cb-embeddings
$$C_p \hookrightarrow (R \oplus_2 OH)/(R \cap_2 \ell^{oh}_2(\lambda))^\perp\,\, \text{and}\,\,
R_p \hookrightarrow (C \oplus_2 OH)/(C \cap_2 \ell^{oh}_2(\lambda))^\perp$$
obtained by a generalized version of ``Pisier's exercise", where $\ell^{oh}_2(\lambda)$ means
the operator Hilbert space on the weighted $\ell_2$ space with respect to the weight $\lambda^2$ for a sequence of strictly positive real numbers
$\lambda = (\lambda_k)_{k\geq 1}$.

The next step is to consider a diagonal operator $\ds_{\lambda^4} = \sum_k \lambda^4_k e_{kk}$,
which can be regarded as the density $D_\psi$ associated to a normal strictly semifinite faithful (n.s.s.f. in short)
weight $\psi$ on $B(\ell_2)$. Let $q_n$ be the projection $\sum_{k\leq n}e_{kk}$ and $\psi_n$ be
the restriction of $\psi$ to the subalgebra $M_n = q_nB(\ell_2)q_n$.
Now we set $$\ku_n = \psi_n(q_n) = \sum^n_{k=1}\lambda^4_k,$$ and let $\varphi_n$ and $\tphi_n$ be states on $M_n$
and $M_n \oplus M_n$, respectively, defined by
$$\varphi_n = \psi_n/\ku_n\,\, \text{and}\,\, \tphi_n(x,y) = \frac{1}{2}(\varphi_n(x)+\varphi_n(y))$$
for $x,y \in M_n$.

If $\ku_n$ is an integer, then we have a nice embedding of
$$\K_{1,2}(\psi_n) = [(R_n \oplus_2 OH_n)/(R_n \cap_2 \ell^{oh}_2(\lambda))^\perp]
\otimes_h [(C_n \oplus_2 OH_n)/(C_n \cap_2 \ell^{oh}_2(\lambda))^\perp]$$ as follows.

\begin{prop}\label{prop-Rp-Embedding2}
Assume that $\text{\upshape{k}}_n = \sum^n_{k=1}\lambda^4_k$ is an integer and define
$$\A_n = *^{\ku_n}_{j=1}(M_n \oplus M_n, \tphi_n).$$
If $\pi_j : M_n \oplus M_n \rightarrow \A_n$ is the natural embedding into the $j$-th component of $\A_n$, then the mapping
$$w_n : \K_{1,2}(\psi_n) \rightarrow L_1(\A_n ; OH_{\ku_n}), \,\, x \mapsto
\frac{1}{\ku_n}\sum^{\ku_n}_{j=1}\pi_j (x, -x)\otimes e_j$$ is a cb-embedding with constants independent of $n$.
\end{prop}
\begin{proof}
See Lemma 2.11 of \cite{JP-LpLq}.
\end{proof}

Combining with Proposition \ref{prop-OH-Embedding} we get an embedding
$\K_{1,2}(\psi_n)\hookrightarrow L_1(\A_n \overline{\otimes} \M_{\ku_n})$
by $(I_{L_1(\A_n)}\otimes v_{\ku_n})\circ w_n$.
Now we consider the embedding for $$\K_{1,2}(\psi)= [(R \oplus_2 OH)/(R \cap_2 \ell_2(\lambda)^{oh})^\perp]
\otimes_h [(C \oplus_2 OH)/(C \cap_2 \ell_2(\lambda)^{oh})^\perp].$$
Note that we may assume that $\ku_n = \sum^n_{k=1}\lambda^4_k$'s are non-decreasing positive integers
since we may approximate each $\ku_n$ by its closest integer.
This allows us to recover $\K_{1,2}(\psi)$ by a completely isometric embedding
$$\K_{1,2}(\psi)= \overline{\cup_{n\geq 1}\K_{1,2}(\psi_n)} \hookrightarrow \prod_{n, \U}\K_{1,2}(\psi_n).$$
Thus, according to \cite{Ray02} we get a cb-embedding
$$\K_{1,2}(\psi) \hookrightarrow L_1(\B)\,\, \text{with}\,\, \B = \Big(\prod_{n,\U}(\A_n \overline{\otimes} \M_{\ku_n})_*\Big)^*,$$
and by the stability of $QWEP$ with respect to free product, tensor product and ultraproduct (\cite{J02, J05}) $\B$ also satisfies $QWEP$.

However the embedding above is not appropriate for our purpose, since we do not know whether $\K_{1,2}(\psi)$ itself
is cb-complemented in $L_1(\B)$ or not, so that we need to find another embedding of $\K_{1,2}(\psi)$
which is cb-complemented in the noncommutative $L_1$ space with respect to a von Neumann algebra with $QWEP$. 
We will use the following noncommutative version of Rosenthal's inequality
for identically distributed random variables in $L_1$ from \cite{J-Araki} and \cite{JP-LpLq}.

Let $\N$ and $\A$ be $\sigma$-finite von Neumann algebras
with a normal faithful conditional expectation $\Es_\N : \A \rightarrow \N$.
We recall that a family of von Neumann algebras $(\As_k)_{k\geq 1}$ satisfying $\N \subseteq \As_k \subseteq \A$
is a system of {\it symmetrically independent copies over} $\N$ ({\it s.i.c.} in short) when

\begin{itemize}
\item[\bf{(i)}] If $a \in \left\langle \As_1, \cdots, \As_{k-1}, \As_{k+1}, \cdots \right\rangle$
and $b\in \As_k$, then we have $$\Es_\N(ab) = \Es_\N(a)\Es_\N(b).$$

\item[\bf{(ii)}] There is a von Neumann algebra $\As$ containing $\N$, a normal faithful conditional expectation
$\Es_0 : \As \rightarrow \N$ and isomorphisms $\pi_k : \As \rightarrow \As_k$ such that $$\Es_\N \circ \pi_k = \Es_0$$
and the following holds for every permutation $\alpha$ of the integers
$$\Es_\N(\pi_{j_1}(a_1)\cdots \pi_{j_m}(a_m)) = \Es_\N(\pi_{\alpha(j_1)}(a_1)\cdots \pi_{\alpha(j_m)}(a_m)).$$

\item[\bf{(iii)}] There is a normal faithful conditional expectation $\mathcal{E}_k : \A \rightarrow \As_k$
such that $$\Es_\N = \Es_0 \pi_k^{-1}\mathcal{E}_k.$$
\end{itemize}

\begin{prop}\label{prop-Rosenthal}
Let $\N$, $\A$ and $(\As_k)_{k\geq 1}$ are as before and $(\As_k)_{k\geq 1}$ is a system of s.i.c. over $\N$.
Then for $x\in L_1(\As)$ with $\Es_0(x) = 0$ we have
$$\norm{\sum^n_{k=1}\pi_k(x)}_{L_1(\A)} \sim \inf_{x = x_1 + x_2 + x_3} n\norm{x_1}_{L_1(\As)}
+ n^{\frac{1}{2}}\norm{x_2}_{L^r_1(\As, \Es_0)} + n^{\frac{1}{2}}\norm{x_3}_{L^c_1(\As, \Es_0)}.$$
\end{prop}
\begin{proof}
See Theorem 6.11 of \cite{J-Araki} and Lemma 4.9 of \cite{JP-LpLq}.
\end{proof}

Now we turn our attention back to $\K_{1,2}(\psi_n)$ and assume that $\text{\upshape{k}}_n = \sum^n_{k=1}\lambda^4_k$
is an integer as before. Then it is clear that $$(\pi_j(M_n \oplus M_n)\overline{\otimes}\rho_j(M))^{\ku_n}_{j=1}$$
is {\it s.i.c. over $\mathbb{C}$} with
$$\A = \A_n \overline{\otimes} \M_{\ku_n},\,\, \As = (M_n\oplus M_n) \overline{\otimes} M,$$
$$\Es_{\mathbb{C}}= *^{\ku_n}_{j=1}\tphi_n \otimes *^{\ku_n}_{j=1}\phi \,\,\text{and}\,\,
\Es_0 = \tphi_n \otimes \phi,$$ where $M$ and $\phi$ are from \eqref{notation1}, so that we can calculate the norm of the image of
$(I_{L_1(\A_n)}\otimes v_{\ku_n})\circ w_n$ as follows.

\begin{prop}\label{prop-NormCalculation}
Assume that we are in the same situation as in Proposition \ref{prop-Rp-Embedding2}
and let $\gamma_1 = \sum_{k\in \z}D_{\phi}^{\frac{1}{2}}g_{k}D_{\phi}^{\frac{1}{2}} \in L_1(M)$, where $M$, $\phi$ and $(g_k)_{k\in \z}$ 
are from \eqref{notation1}. Then for $x\in L_1(M_n)$ we have
\begin{align*}
\begin{split}
\lefteqn{\norm{\sum^{\ku_n}_{k=1}\pi_j(x, -x)\otimes \rho_j(\gamma_1)}_{L_1(\A)}}\\
& \sim \inf_{x\otimes \gamma_1 = x_1 + x_2 + x_3} \ku_n\norm{x_1}_{L_1(\As')}
+ \ku_n^{\frac{1}{2}}\norm{x_2}_{L^r_1(\As', \Es_1)} + \ku_n^{\frac{1}{2}}\norm{x_3}_{L^c_1(\As', \Es_1)},
\end{split}
\end{align*}
where $\As' = M_n \overline{\otimes} M$ and $\Es_1 = \varphi_n \otimes \phi$.
\end{prop}
\begin{proof}
This is a direct application of Proposition \ref{prop-Rosenthal} taking the completely contractive map
$L_1(\As) \rightarrow L_1(\As'), \,\, (x,y) \mapsto \frac{1}{2}(x-y)$ into account.
\end{proof}

Now we consider a cb-embedding of $\K_{1,2}(\psi_n)$ into
$$\K^1_{RC_1}(\psi_n \otimes \phi) = \ku_n L_1(\As') + \ku_n^{\frac{1}{2}} L^{r_1}_2(\As')
+ \ku_n^{\frac{1}{2}} L^{c_1}_2(\As').$$
More precisely, we have $$\norm{x}_{S^m_1(\K^1_{RC_1}(\psi_n \otimes \phi))} = \inf\Big\{ \ku_n\norm{x_1}_{S^m_1(L_1(\As'))}
+  \ku_n^{\frac{1}{2}}\norm{x_2}_{S^m_1(L^{r_1}_2(\As'))} + \ku_n^{\frac{1}{2}}\norm{x_3}_{S^m_1(L^{c_1}_2(\As'))} \Big\},$$
where the infimum runs over all possible decompositions
$$x = x_1 + (I_{S^m_1}\otimes D_{\varphi_n \otimes \phi}^{\frac{1}{2}})x_2 +
x_3(I_{S^m_1}\otimes D_{\varphi_n \otimes \phi}^{\frac{1}{2}}),$$ where $D_{\varphi_n \otimes \phi}$ is the density of $\varphi_n \otimes \phi$.

\begin{thm}\label{thm-Rp-Embedding4}
Assume that we are in the same situation as in the Proposition \ref{prop-NormCalculation},
then the mapping $$u_n : \K_{1,2}(\psi_n) \rightarrow \K^1_{RC_1}(\psi_n \otimes \phi), \,\,
x \mapsto \frac{1}{\ku_n}x \otimes \gamma_1$$ is a cb-embedding with constants independent of $n$.
Furthermore, $\K^1_{RC_1}(\psi_n \otimes \phi)$ is completely complemented in
$L_1(*^{\ku_n}_{j=1}(\As' \oplus \As'))$ with constants independent of $n$.
\end{thm}
\begin{proof}
We consider $M_m(\A)$, $M_m(\As)$ and $I_{M_m}\otimes \Es_0$ instead of $\A$, $\As$ and $\Es_0$,
respectively, and apply Proposition \ref{prop-Rosenthal} taking the contractive map
$$S^m_1(L_1(\As)) \rightarrow S^m_1(L_1(\As')), \,\, (x,y) \mapsto \frac{1}{2}(x-y)$$ into account.
Note that we have
	\begin{align*}
	\norm{(I_{S^m_1}\otimes D_{\varphi_n \otimes \phi}^{\frac{1}{2}})a}_{S^m_1(L^{r_1}_2(\As'))} & =
	m \norm{(I_{S^m_1} \otimes D_{\varphi_n \otimes \phi})a}_{L^r_1(M_m(\As'), I_{M_m}\otimes \Es_1)}
	\end{align*}
and
	\begin{align*}
	\norm{b(D_{\varphi_n \otimes \phi}^{\frac{1}{2}} \otimes I_{S^m_1})}_{S^m_1(L^{c_1}_2(\As'))}
	& = m\norm{b(D_{\varphi_n \otimes \phi}\otimes I_{S^m_1})}_{L^{c}_1(M_m(\As'), I_{M_m}\otimes \Es_1)}.
	\end{align*}
The second statement is from Corollary 7.10 of \cite{J05}.
\end{proof}

\begin{rem}{\rm
The above approach is the same as that of \cite{JP-LpLq}, which was used in constructing the embedding of $S_p$ into
the predual of a hyperfinite von Neumann algebra. However, we are using $\A_n$, the free product of $M_n \oplus M_n$
to be consistent with Proposition \ref{prop-Rp-Embedding2} instead of the tensor product of $M_n \oplus M_n$.}
\end{rem}

We can describe the operator space structure of $u_n(\K_{1,2}(\psi_n))$ more precisely. Let
\begin{align*}
\K^\delta_n & = R_n\widehat{\otimes}\,C_n \widehat{\otimes}\,(\ell^r_2(\delta^{\frac{k}{2}}) +\ell^c_2(\delta^{-\frac{k}{2}}))
\\ &  \;\;\;\; + R_n \widehat{\otimes}\,\ell^r_2(\lambda^{-2})\,\widehat{\otimes}\,
\ell^r_2(\delta^{\frac{k}{2}}) + \ell^c_2(\lambda^{-2})\, \widehat{\otimes}\, C_n \widehat{\otimes}\,
\ell^c_2(\delta^{-\frac{k}{2}}) \\ & = \K^\delta_n(L_1) + \K^\delta_n(r) + \K^\delta_n(c),
\end{align*}
where $\lambda^{-2}$ means the sequence $(\lambda^{-2}_k)_{k\geq 1}$.

\begin{prop}\label{prop-Rp-Embedding5}
Assume that we are in the same situation as in the Proposition \ref{prop-NormCalculation}. 
Let $1 <\delta \leq 2$ and $P \;(= P_\delta) : L_1(M) \rightarrow G^1_*$ be the canonical projection onto $G^1_*$. Then
$$(I_{S^n_1}\otimes P)\, \K^1_{RC_1}(\psi_n \otimes \phi)
\rightarrow \K^\delta_n,\,\, \frac{1}{\ku_n}x \otimes D^{\frac{1}{2}}_{\phi}g_k D^{\frac{1}{2}}_{\phi}
\mapsto x\otimes e_k$$ is a complete isomorphism with constants independent of $\delta$ and $n$.
\end{prop}
\begin{proof}
Let $\As' = M_n \overline{\otimes} M$. Then, for $x \in S^m_1(\K^1_{RC_1}(\psi_n \otimes \phi))$ we have
\begin{align*}
\lefteqn{\frac{1}{\ku_n}\norm{x}_{S^m_1(\K^1_{RC_1}(\psi_n \otimes \phi))}}\\ & =
\inf \Big\{ \norm{x_1}_{S^m_1(L_1(\As'))} +  \ku_n^{-\frac{1}{2}}\norm{x_2}_{S^m_1(L^{r_1}_2(\As'))}
+ \ku_n^{-\frac{1}{2}}\norm{x_3}_{S^m_1(L^{c_1}_2(\As'))} \\
& \;\;\;\;\;\;\;\;\;\;\;\; : x = x_1 + (I_{S^m_1}\otimes D_{\varphi_n \otimes \phi}^{\frac{1}{2}})x_2 +
x_3(I_{S^m_1}\otimes D_{\varphi_n \otimes \phi}^{\frac{1}{2}}) \Big\}\\
& = \inf \Big\{ \norm{y_1}_{S^m_1(L_1(\As'))} +  \ku_n^{-\frac{1}{2}}\norm{(I_{S^m_1}\otimes D_{\varphi_n \otimes \phi}^{-\frac{1}{2}})y_2}_{S^m_1(L^{r_1}_2(\As'))} \\ & \;\;\;\;\;\;\;\;\;\;\;\;
+ \ku_n^{-\frac{1}{2}}\norm{y_3(I_{S^m_1}\otimes D_{\varphi_n \otimes \phi}^{-\frac{1}{2}})}_{S^m_1(L^{c_1}_2(\As'))}
: x = y_1 + y_2 + y_3 \Big\}\\ & = \inf \Big\{ \norm{y_1}_{S^m_1(L_1(\As'))} +
\norm{(I_{S^m_1}\otimes D_{\psi_n \otimes \phi}^{-\frac{1}{2}})y_2}_{S^m_1(L^{r_1}_2(\As'))} \\ & \;\;\;\;\;\;\;\;\;\;\;\;
+ \norm{y_3(I_{S^m_1}\otimes D_{\psi_n \otimes \phi}^{-\frac{1}{2}})}_{S^m_1(L^{c_1}_2(\As'))}
: x = y_1 + y_2 + y_3 \Big\},
\end{align*}
where $D_{\psi_n \otimes \phi}$ is the density of $\psi_n \otimes \phi$.

Let $y_i = \sum_k y_{i,k}\otimes D^{\frac{1}{2}}_{\phi}g_k D^{\frac{1}{2}}_{\phi}$
for $i=1,2,3$. For the first term we have
	\begin{align*}
	\norm{\sum_k y_{1,k}\otimes D^{\frac{1}{2}}_{\phi}g_kD^{\frac{1}{2}}_{\phi}}_{S^m_1(L_1(\As'))}
	& = \norm{\sum_k y_{1,k}\otimes D^{\frac{1}{2}}_{\phi}g_k D^{\frac{1}{2}}_{\phi}}_{S^m_1(S^n_1(G^1_*))}\\ &
	\sim \norm{\sum_k y_{1,k}\otimes e_k}_{S^m_1 \widehat{\otimes} \K_n(L_1)}.
	\end{align*}

For the second term we recall that $g_k D_{\phi} = \delta^{-2k}D_{\phi}g_k$ by \eqref{modular-condition}, then we have
\begin{align*}
\lefteqn{\norm{(I_{S^m_1}\otimes D_{\psi_n \otimes \phi}^{-\frac{1}{2}})\sum_k y_{2,k}\otimes D^{\frac{1}{2}}_{\phi}g_kD^{\frac{1}{2}}_{\phi}}_{S^m_1(L^{r_1}_2(\As'))}}\\
& = \norm{\sum_k (I_{S^m_1}\otimes \ds_{\lambda^{-2}})y_{2,k}\otimes g_k D^{\frac{1}{2}}_{\phi}}_{S^m_1(S^n_1(L^{r_1}_2(\As')))}\\
& = \norm{(I_{S^m_1}\otimes \text{tr}_{\As'})\Big(\sum_{k,l} (I_{S^m_1}\otimes \ds_{\lambda^{-2}})y_{2,k} y_{2,l}^* (I_{S^m_1}\otimes \ds_{\lambda^{-2}})
\otimes  g_k D_{\phi} g^*_l \Big)^{\frac{1}{2}}}_{S^m_1}\\
& = \norm{(I_{S^m_1}\otimes \text{tr}_{\As'})\Big(\sum_{k,l} (I_{S^m_1}\otimes \ds_{\lambda^{-2}})y_{2,k} y_{2,l}^* (I_{S^m_1}\otimes \ds_{\lambda^{-2}})
\otimes \delta^{-2k} D_{\phi} g_k  g^*_l  \Big)^{\frac{1}{2}}}_{S^m_1}\\
& = \norm{\Big(\sum_{k,l} (I_{S^m_1}\otimes \ds_{\lambda^{-2}})y_{2,k} y_{2,l}^* (I_{S^m_1}\otimes \ds_{\lambda^{-2}})
\phi(g_k g^*_l) \delta^{-2k} \Big)^{\frac{1}{2}}}_{S^m_1}\\
& = \norm{\Big(\sum_{k} (I_{S^m_1}\otimes \ds_{\lambda^{-2}})y_{2,k} y_{2,k}^* (I_{S^m_1}\otimes \ds_{\lambda^{-2}})
\delta^{-k}\Big)^{\frac{1}{2}}}_{S^m_1}\\
& = \norm{\sum_{k} (I_{S^m_1}\otimes \ds_{\lambda^{-2}})y_{2,k}\otimes
\delta^{-\frac{k}{2}} e_k}_{S^m_1 \widehat{\otimes} L^c_2(M_n) \widehat{\otimes} \ell^c_2(\z)}
= \norm{\sum_{k} y_{2,k}\otimes e_k}_{S^m_1 \widehat{\otimes} \K_n(c)},
\end{align*}
where $\ds_{\lambda^{-2}}$ is the diagonal operator $\sum_k \lambda^{-2k} e_{kk}$.

Similarly, we have
	$$\norm{\sum_k y_{3,k}\otimes D^{\frac{1}{2}}_{\phi}g_kD^{\frac{1}{2}}_{\phi}(I_{S^m_1}\otimes D_{\psi_n \otimes
	\phi}^{-\frac{1}{2}})}_{S^m_1(L^{c_1}_2(\As'))} = \norm{\sum_{k} y_{3,k}\otimes e_k}_{S^m_1 \widehat{\otimes} \K_n(r)}.$$
\end{proof}

Let's consider $\K_{1,2}(\psi)$ again. Then, we may assume that $\ku_n = \sum^n_{k=1}\lambda^{4}_k$'s are non-decreasing positive integers as before.
This allows us to recover $\K_{1,2}(\psi)$ by a completely isometric embedding
$$\K_{1,2}(\psi)= \overline{\cup_{n\geq 1}\K_{1,2}(\psi_n)} \hookrightarrow \prod_{n, \U}\K_{1,2}(\psi_n).$$
Thus, according to \cite{Ray02} we get a cb-embedding
$$\K_{1,2}(\psi) \hookrightarrow \prod_{n, \U}(I_{S^n_1}\otimes P)\K^1_{RC_1}(\psi_n \otimes \phi)
\subseteq L_1(\Bs)\,\, \text{with}\,\, \Bs = \Big(\prod_{n,\U}(*^{\ku_n}_{j=1}(\As' \oplus \As'))_*\Big)^*,$$
where $\As' = M_n \overline{\otimes} M$, and by the stability of $QWEP$ with respect to free product,
tensor product and ultraproduct $\Bs$ also satisfies $QWEP$. Moreover, since each
$(I_{S^n_1}\otimes P)\K^1_{RC_1}(\psi_n \otimes \phi)$ is cb-complemented in $L_1(*^{\ku_n}_{j=1}(\As' \oplus \As'))$
with uniformly bounded cb-norms $\prod_{n, \U}(I_{S^n_1}\otimes P)\K^1_{RC_1}(\psi_n \otimes \phi)$
is also cb-complemented in $L_1(\Bs)$.

Furthermore, by Proposition \ref{prop-Rp-Embedding5} we have the cb-isomorphism
\begin{align}\label{density1}
\K_{1,2}(\psi) & \cong 
R \, \widehat{\otimes}\, C \, \widehat{\otimes}\,(\ell^r_2(\delta^{\frac{k}{2}}) +\ell^c_2(\delta^{-\frac{k}{2}}))
\\ &  \;\;\;\; + R \, \widehat{\otimes}\, \ell^r_2(\lambda^{-2})\, \widehat{\otimes}\,
\ell^r_2(\delta^{\frac{k}{2}}) + \ell^c_2(\lambda^{-2})\, \widehat{\otimes}\, C \, \widehat{\otimes}\,
\ell^c_2(\delta^{-\frac{k}{2}}). \nonumber
\end{align}

\section{The change of density}\label{sec-ChangeOfDensity}

In this section we present a concrete embedding of $\Pi^o_1(OH, S_p)$ using the materials in the previous section.
As was pointed out in Section \ref{sec-dualProb} we need to consider embeddings of $OH$ and $S_p$.
In the case of $OH$ we have by Proposition \ref{prop-OH-Embedding}
$$v^\delta_n : OH \rightarrow G^{\n}_* \subseteq L_1(\M_{\n}), \,\, e_j \mapsto
\sum_{k\in \z}D_{\Phi}^{\frac{1}{2}}g_{k,j}D_{\Phi}^{\frac{1}{2}}$$ for a fixed $\delta=2$. Moreover,
$G^{\n}_*$ is 2-completely complemented in $L_1(\M_{\n})$ and cb-isomorphic to
$$L^c_2(t^{-\frac{1}{2}}; \ell_2) + L^r_2(t^{\frac{1}{2}}; \ell_2).$$

Now we consider the case of $S_p$. Then we start with the observation
\begin{align}\label{density2}
S_p = C_p \otimes_h R_p & \hookrightarrow K_{c, \theta}\otimes_h K_{r, \theta}\\
& = \Big(L^r_2(t^{-\theta}; \ell_2) + L^{oh}_2(t^{1-\theta}; \ell_2)\Big) 
\otimes_h \Big( L^c_2(s^{-\theta}; \ell_2) + L^{oh}_2(s^{1-\theta}; \ell_2)\Big). \nonumber
\end{align}
Thus, we need to consider the situation $(R+\ell^{oh}(\lambda)) \otimes_h (C+\ell^{oh}(\lambda))$ by a suitable identification.
However, we have
\begin{align*}
\norm{x}_{M_m(R+\ell^{oh}_2(\lambda))} & \sim \inf_{x = x_1 + x_2} \norm{x_1}_{M_m(R)}
+ \norm{x_2(I_{M_m}\otimes\ds_{\lambda})}_{M_m(OH)}\\ & = \inf_{x = y_1 + y_2(I_{M_m}\otimes\ds^{-1}_{\lambda})}
\norm{y_1}_{M_m(R)} + \norm{y_2}_{M_m(OH)}\\ & \sim \norm{x}_{M_m((R \oplus_2 OH)/(R \cap_2 \ell^{oh}_2(\lambda^{-1}))^\perp)}
\end{align*}
and similarly $\norm{x}_{M_m(C+\ell^{oh}(\lambda))} \sim \norm{x}_{M_m((C \oplus_2 OH)/(C \cap_2 \ell^{oh}_2(\lambda^{-1}))^\perp)}$ for any $m\in \n$.
Thus, we have a complete isomorphism
\begin{equation}\label{density3}
(R+\ell^{oh}(\lambda))\otimes_h (C + \ell^{oh}(\lambda)) \cong \K_{1,2}(\psi^{-1}),
\end{equation}
where $\psi^{-1}$ is the weight associated to $\sum_k \lambda^{-4}_k e_{kk}$.
By combining \eqref{density1}, \eqref{density2} and \eqref{density3} we can guess that $S_p$ can be embedded in the space $\K_{S_p}$ defined by
	\begin{align}\label{changed-density}
	\K_{S_p} & = L^r_2(s^{-\theta}; \ell_2)\widehat{\otimes}L^c_2(t^{-\theta}; \ell_2)
	\widehat{\otimes}L^r_2(u^{\frac{1}{2}}) + L^r_2(s^{-\theta}; \ell_2)\widehat{\otimes}L^c_2(t^{-\theta}; \ell_2)
	\widehat{\otimes} L^c_2(u^{-\frac{1}{2}})\\
	&  \;\;\;\; + L^r_2(s^{-\theta}; \ell_2)\widehat{\otimes}L^r_2(t^{2-\theta}; \ell_2)
	\widehat{\otimes} L^r_2(u^{\frac{1}{2}}) + L^c_2(s^{2-\theta}; \ell^n_2)\widehat{\otimes}L^c_2(t^{-\theta}; \ell_2)
	\widehat{\otimes} L^c_2(u^{-\frac{1}{2}}), \nonumber
	\end{align}
which is a 4-term sum of vector valued function space with 3 variables $(s,t,u) \in \mathbb{R}^3_+$.
It is worth of mention that we can observe a nontrivial change of density between \eqref{density2} and \eqref{changed-density}.

\begin{thm}\label{thm-Sp-Embedding}
Let $1 < p <2$, $\frac{1}{p} + \frac{1}{p'}= 1$ and $\theta = \frac{2}{p'}$. Then we have the following cb-embedding
$$C_p \otimes_h R_p \rightarrow \K_{S_p},\; e_{i1} \otimes e_{1j} \mapsto ({\bf 1} \otimes e_i)\otimes ({\bf 1} \otimes e_j) \otimes {\bf 1}.$$
More precisely, for any $m \in \n$ we have
\begin{align*}
\lefteqn{\norm{\sum^n_{i,j=1}x_{ij}\otimes e_{i1} \otimes e_{1j}}_{M_m(C_p \otimes_h R_p)}}\\
& \sim \theta(1-\theta)\norm{\sum^n_{i,j=1}x_{ij}\otimes ({\bf 1} \otimes e_i) \otimes ({\bf 1} \otimes e_j) \otimes {\bf 1}}_{M_m(\K_{S_p})}.
\end{align*}
Moreover, $\K_{S_p}$ is completely complemented in the noncommutative $L_1$ space with respect to a von Neumann algebra with $QWEP$.
\end{thm}
\begin{proof}
For $1<\delta \leq 2$ and $\alpha \in \mathbb{R}$ we consider the following maps 
	$$\Phi_{\delta, \alpha} : \ell_2(\delta^{\alpha k}) \rightarrow L_2(t^\alpha), \;
	(x_k)_{k\in \z} \mapsto (\log \delta)^{-\frac{1}{2}}\sum_{k\in \z}x_k{\bf 1}_{[\delta^k, \delta^{k+1})}(t)$$
and 
	$$\Psi_{\delta, \alpha} : L_2(t^\alpha) \rightarrow \ell_2(\delta^{\alpha k}), \;
	f \mapsto \Big((\log \delta)^{-\frac{1}{2}}\int^{\delta^{k+1}}_{\delta^k} f(t) \frac{dt}{t}\Big)_{k\in \z}.$$
Then we have $\Psi_{\delta, \alpha} \circ \Phi_{\delta, \alpha} = I_{\ell_2(\delta^{\alpha k})}$ and
	$$\norm{\Phi_{\delta, \alpha}} \leq \max(1,\delta^\alpha)\; \text{and}\; \norm{\Psi_{\delta, \alpha}} \leq \max(1,\delta^{-\alpha}).$$
Note that $\Phi_{\delta, \alpha}$ (resp. $\Psi_{\delta, \alpha}$) is uniformly bounded for $-1 < \alpha < 2$
(In particular, for $\alpha \in \{ -\theta, (1-\theta), (2-\theta) \}$),
and it is actually the same map regardless of $\alpha$, so that we just denote by $\Phi_{\delta}$ and $\Psi_{\delta}$.

Now we fix $m \in \n$ and $x \in M_m(C_p \otimes_h R_p)$. Since $\cup_{1< \delta \leq 2}\{\text{ran}\Phi_{\delta, \alpha}\}$
is dense in $L_2(t^\alpha)$ we can choose $1< \delta \leq 2$ with $\delta -1$ small enough so that there is 
$$y = I_{M_m} \otimes \Big[(\Phi_\delta \otimes I_{\ell_2})\otimes (\Phi_\delta \otimes I_{\ell_2})\Big](z)
\in M_m(K_{c, \theta}\otimes_h K_{r, \theta})$$ with very small
$$\norm{\sum^n_{i,j=1}x_{ij}\otimes ({\bf 1} \otimes e_i)\otimes ({\bf 1} \otimes e_j) - y}_{M_m(K_{c, \theta}\otimes_h K_{r,\theta})}$$ and
$$\norm{\sum^n_{i,j=1}x_{ij}\otimes ({\bf 1} \otimes e_i)\otimes ({\bf 1} \otimes e_j) \otimes {\bf 1} - y \otimes {\bf 1}}_{M_m(\K_{S_p})},$$
where $z \in M_m(B_\delta)$ and $$B_\delta = \Big(\ell^r_2(\delta^{-\theta k}; \ell_2) + \ell^{oh}_2(\delta^{(1-\theta)k}; \ell_2)\Big) 
\otimes_h \Big( \ell^c_2(\delta^{-\theta k}; \ell_2) + \ell^{oh}_2(\delta^{(1-\theta)k}; \ell_2)\Big).$$
By applying \eqref{density1} (in this case $(\delta^{2k})_{k\in \z}$ is the weight) and \eqref{density3} to $B_\delta$ we get the following cb-embedding with constant independent of $\delta$. 
\begin{align*}
B_\delta & \hookrightarrow C_{\delta} = \ell^r_2(\delta^{-\theta k}; \ell_2) \widehat{\otimes}\, \ell^c_2(\delta^{-\theta k}; \ell_2) \widehat{\otimes} \Big(\ell^r_2(\delta^{\frac{k}{2}}) + \ell^c_2(\delta^{-\frac{k}{2}})\Big)\\& \;\;\;\;\;\;\;\;\;\;\;\;\;
+ \ell^r_2(\delta^{-\theta k}; \ell_2) \widehat{\otimes}\, \ell^r_2(\delta^{(2-\theta) k}; \ell_2) \widehat{\otimes}\, \ell^r_2(\delta^{\frac{k}{2}})
\\ & \;\;\;\;\;\;\;\;\;\;\;\;\; + \ell^c_2(\delta^{(2-\theta)k}; \ell_2)\widehat{\otimes}\, \ell^c_2(\delta^{-\theta k}; \ell_2) \widehat{\otimes} \, \ell^c_2(\delta^{-\frac{k}{2}}),\\ w & \mapsto w \otimes \sum_{k\in \z}e_k.
\end{align*}
Note that ${\bf 1} = \Phi_{\delta}(\sum_{k\in \z}e_k)$ and
$$(\Phi_\delta \otimes I_{\ell_2})\otimes (\Phi_\delta \otimes I_{\ell_2})\otimes \Phi_\delta : C_\delta \rightarrow \K_{S_p}$$ and
$$(\Psi_\delta \otimes I_{\ell_2})\otimes (\Psi_\delta \otimes I_{\ell_2}) : K_{c, \theta}\otimes_h K_{r, \theta} \rightarrow B_\delta$$
are cb-maps with uniformly bounded cb-norms, so by Proposition \ref{prop-Rp-C-OH-embedding} we have
\begin{align*}
\lefteqn{\theta^{-1}(1-\theta)^{-1}\norm{\sum^n_{i,j=1}x_{ij}\otimes e_{i1} \otimes e_{1j}}_{M_m(C_p \otimes_h R_p)} } & \\
& \sim \norm{\sum^n_{i,j=1}x_{ij}\otimes ({\bf 1} \otimes e_i) \otimes ({\bf 1} \otimes e_j)}_{M_m(K_{c, \theta}\otimes_h K_{r, \theta})}\\ & \sim \norm{y}_{M_m(K_{c, \theta}\otimes_h K_{r, \theta})} \sim \norm{z}_{M_m(B_\delta)} \sim \norm{z \otimes \sum_{k\in \z}e_k}_{M_m(C_{\delta})}\\
& \sim \norm{y \otimes {\bf 1}}_{M_m(\K_{S_p})}
\sim \norm{\sum^n_{i,j=1}x_{ij}\otimes ({\bf 1} \otimes e_i) \otimes ({\bf 1} \otimes e_j) \otimes {\bf 1}}_{M_m(\K_{S_p})}.
\end{align*}
Note that all equivalences above are independent of the choice of $\delta$.

Moreover, for any $1< \delta \leq 2$
$$E_\delta = \Big[(\Phi_{\delta} \otimes I_{\ell_2})\otimes (\Phi_{\delta} \otimes I_{\ell_2})\otimes \Phi_{\delta}\Big](C_\delta) \cong C_\delta$$
completely isometrically and by Proposition \ref{prop-Rp-Embedding5} and the following argument we have a cb-embedding
$$C_\delta \hookrightarrow D_\delta \subseteq L_1(N_\delta),$$
where $N_\delta$ satisfies QWEP and $D_\delta$ is completely complemented in $L_1(N_\delta)$ with constants independent of $\delta$.

Let $\U'$ be a free ultrafilter on the collection of subsets of $(1,2]$ containing all $(1,\delta]$ for $1< \delta \leq 2$. Then we have
$$\K_{S_p} = \overline{\cup_{1<\delta\leq 2}E_\delta} \hookrightarrow \prod_{\delta, \U'}D_\delta 
\subseteq L_1(\Cs),\;\text{with}\; \Cs = \Big(\prod_{\delta, \U'}L_1(N_\delta)\Big)^*.$$
By the stability of $QWEP$ with respect to free product, tensor product and ultraproduct $\Cs$ also satisfies $QWEP$. Moreover, since each
$D_\delta$ is cb-complemented in $L_1(N_\delta)$ with uniformly bounded cb-norms $\prod_{\delta, \U'}D_\delta$
is also cb-complemented in $L_1(\Cs)$.

\end{proof}

By combining the above two embeddings for $OH$ and $S_p$ we get an embedding of $\Pi^o_1(OH_n, S^n_p)$ to the following space
$\K_{\Pi^o_1(OH_n, S^n_p)}$, which is a 8-term sum of vector valued function space with 4 variables $(s,t,u,v) \in \mathbb{R}^4_+$!
Let $\K_{S^n_p}$ be the space $\K_{S_p}$ using $\ell^n_2$ instead of $\ell_2$. Then we define
\begin{align*}
\K_{\Pi^o_1(OH_n, S^n_p)} & = \K_{S^n_p}\widehat{\otimes}(L^c_2(v^{-\frac{1}{2}}; \ell^n_2) +_2 L^r_2(v^{\frac{1}{2}}; \ell^n_2))\\
& =  L^c_2(s^{2-\theta}; \ell^n_2)\widehat{\otimes}L^c_2(t^{-\theta}; \ell^n_2)
\widehat{\otimes} L^c_2(u^{-\frac{1}{2}})\widehat{\otimes}L^c_2(v^{-\frac{1}{2}}; \ell^n_2)\\
&  \;\;\;\;+ L^r_2(s^{-\theta}; \ell^n_2)\widehat{\otimes}L^r_2(t^{2-\theta}; \ell^n_2)
\widehat{\otimes} L^r_2(u^{\frac{1}{2}})\widehat{\otimes}L^r_2(v^{\frac{1}{2}}; \ell^n_2)\\
&  \;\;\;\;+ L^c_2(s^{2-\theta}; \ell^n_2)\widehat{\otimes}L^c_2(t^{-\theta}; \ell^n_2)
\widehat{\otimes} L^c_2(u^{-\frac{1}{2}})\widehat{\otimes}L^r_2(v^{\frac{1}{2}}; \ell^n_2)\\
&  \;\;\;\;+ L^r_2(s^{-\theta}; \ell^n_2)\widehat{\otimes}L^r_2(t^{2-\theta}; \ell^n_2)
\widehat{\otimes} L^r_2(u^{\frac{1}{2}})\widehat{\otimes}L^c_2(v^{-\frac{1}{2}}; \ell^n_2)\\
&  \;\;\;\;+ L^r_2(s^{-\theta}; \ell^n_2)\widehat{\otimes}L^c_2(t^{-\theta}; \ell^n_2)
\widehat{\otimes} L^c_2(u^{-\frac{1}{2}})\widehat{\otimes} L^c_2(v^{-\frac{1}{2}}; \ell^n_2)\\
&  \;\;\;\;+ L^r_2(s^{-\theta}; \ell^n_2)\widehat{\otimes}L^c_2(t^{-\theta}; \ell^n_2)
\widehat{\otimes}L^c_2(u^{-\frac{1}{2}})\widehat{\otimes} L^r_2(v^{\frac{1}{2}}; \ell^n_2)\\
&  \;\;\;\;+ L^r_2(s^{-\theta}; \ell^n_2)\widehat{\otimes}L^c_2(t^{-\theta}; \ell^n_2)
\widehat{\otimes} L^r_2(u^{\frac{1}{2}})\widehat{\otimes} L^r_2(v^{\frac{1}{2}}; \ell^n_2)\\
&  \;\;\;\;+ L^r_2(s^{-\theta}; \ell^n_2)\widehat{\otimes}L^c_2(t^{-\theta}; \ell^n_2)
\widehat{\otimes}L^r_2(u^{\frac{1}{2}})\widehat{\otimes} L^c_2(v^{-\frac{1}{2}}; \ell^n_2).
\end{align*}

\begin{cor}\label{cor-SptimesOH-Embedding}
Let $1 < p <2$, $\frac{1}{p} + \frac{1}{p'}= 1$ and $\theta = \frac{2}{p'}$. Then we have
the following cb-embedding with constants independent of $n$.
$$\Pi^o_1(OH_n, S^n_p) \rightarrow \K_{\Pi^o_1(OH_n, S^n_p)}, \,\, T_{e_k \otimes e_{ij}} \mapsto
({\bf 1} \otimes e_i)\otimes ({\bf 1} \otimes e_j) \otimes ({\bf 1} \otimes e_k) \otimes {\bf 1}.$$
Moreover, for $a = \sum^n_{i,j,k=1}a_{i,j,k} e_k  \otimes e_{ij} \in OH_n\otimes S^n_p$ we have
\begin{align*}
\pi^o_1(T_a) \sim \theta(1-\theta)\norm{{\bf 1}\otimes a}_{\K_{\Pi^o_1(OH_n, S^n_p)}}.
\end{align*}
\end{cor}

\section{A result for the identity}\label{sec-identity}

In this section we calculate the $\K_{\Pi^o_1(OH_n, S^n_p)}$-norm of ${\bf 1} \otimes \sum^n_{i=1}e_i \otimes e_i \otimes \delta_i$
which corresponds to the formal identity map $I_n : OH_n \rightarrow \ell^n_p$. First, we rearrange
$\K_{\Pi^o_1(OH_n, S^n_p)}$ as follows.
\begin{align*}
\lefteqn{\K_{\Pi^o_1(OH_n, S^n_p)}}\\
& = L^c_2(s^{2-\theta}t^{-\theta}u^{-\frac{1}{2}}v^{-\frac{1}{2}}; \ell^n_2 \otimes_2 \ell^n_2 \otimes_2 \ell^n_2)
\;\;\;\;\;\; + L^r_2(s^{-\theta}t^{2-\theta}u^{\frac{1}{2}}v^{\frac{1}{2}}; \ell^n_2\otimes_2 \ell^n_2\otimes_2 \ell^n_2)\\
&  \;\;\;+ L^c_2(s^{2-\theta}t^{-\theta}u^{-\frac{1}{2}}; \ell^n_2\otimes_2 \ell^n_2)
\widehat{\otimes} L^r_2(v^{\frac{1}{2}}; \ell^n_2)
\,+ L^r_2(s^{-\theta}t^{2-\theta}u^{\frac{1}{2}}; \ell^n_2\otimes_2 \ell^n_2) \widehat{\otimes}L^c_2(v^{-\frac{1}{2}}; \ell^n_2)\\
&  \;\;\;+ L^r_2(s^{-\theta}; \ell^n_2)\widehat{\otimes} L^c_2(t^{-\theta}u^{-\frac{1}{2}}v^{-\frac{1}{2}}; \ell^n_2\otimes_2 \ell^n_2)
+ L^r_2(s^{-\theta}v^{\frac{1}{2}}; \ell^n_2 \otimes_2 \ell^n_2)\widehat{\otimes}L^c_2(t^{-\theta}u^{-\frac{1}{2}}; \ell^n_2)\\
&  \;\;\;+ L^r_2(s^{-\theta}u^{\frac{1}{2}}v^{\frac{1}{2}}; \ell^n_2 \otimes_2 \ell^n_2)\widehat{\otimes}L^c_2(t^{-\theta}; \ell^n_2)
\;\;\;\;\, + L^r_2(s^{-\theta}u^{\frac{1}{2}}; \ell^n_2)\widehat{\otimes}L^c_2(t^{-\theta}v^{-\frac{1}{2}}; \ell^n_2 \otimes_2 \ell^n_2)\\
& = F_1 + F_2 + \cdots + F_8.
\end{align*}
Let $\mu_1$, $\mu_2$ be the measures $$d\mu_1(s,t,u,v) = s^{4-2\theta}t^{-2\theta}u^{-1}v^{-1} \frac{dsdtdudv}{stuv}
\;\text{and}\; d\mu_2(s,t,u,v) = s^{-2\theta}t^{4-2\theta}uv \frac{dsdtdudv}{stuv}$$ corresponding to $F_1$ and $F_2$.
We also let $\mu_{3,1}$ and $\mu_{3,2}$ be the measures $$d\mu_{3,1}(s,t,u) = s^{4-2\theta}t^{-2\theta}u^{-1} \frac{dsdtdu}{stu}\;
\text{and}\; d\mu_{3,2}(v) = v \frac{dv}{v}$$ corresponding to $F_3$, and we
define $\mu_{k,l}$ for $4\leq k \leq 8$ and $l = 1,2$ similarly.

If we look at the Banach space level of $F_l$ it is easier to understand.
For example, we have $$F_1 \cong L_2(\mu_1;\ell^n_2\otimes_2 \ell^n_2\otimes_2 \ell^n_2)$$ and
$$F_3 \cong L_2(\mu_{3,1}; \ell^n_2\otimes_2 \ell^n_2)\otimes_{\pi}L_2(\mu_{3,2}; \ell^n_2)$$ isometrically,
where $\otimes_{\pi}$ implies the projective tensor product in the Banach space category.

In the case of identity we can make the calculation depend only on the decomposition of constant $1$ function by scalar-valued functions.
This will be proved in the following section.
\begin{lem}\label{lem-equi-identity}
\begin{align*}
\lefteqn{\norm{{\bf 1} \otimes \sum^n_{i=1}e_i \otimes e_i \otimes \delta_i}_{\K_{\Pi^o_1(OH_n, S^n_p)}}}\\
& \sim \inf_{{\bf 1} = f_1 + \cdots + f_8} n^{\frac{1}{2}}\norm{f_1}_{L_2(\mu_1)} + n^{\frac{1}{2}}\norm{f_2}_{L_2(\mu_2)}
+ n\norm{f_3}_{L_2(\mu_{3,1})\otimes_{\pi} L_2(\mu_{3,2})} \\& \;\;\;\;\;\;\;\;\;\;\;\;\;\;\;\;\;\;\;\;\;
+ \cdots + n\norm{f_8}_{L_2(\mu_{8,1})\otimes_{\pi} L_2(\mu_{8,2})}.
\end{align*}

\end{lem}

Note that the above infimum is the norm of ${\bf 1}$ in the following function space. 
$$L_2(n\mu_1) + L_2(n\mu_2) + L_2(n\mu_{3,1})\otimes_{\pi} L_2(n\mu_{3,2}) + \cdots + L_2(n\mu_{8,1})\otimes_{\pi} L_2(n\mu_{8,2}).$$
Now we do the calculation for the identity.

\begin{thm}\label{thm-calculation-identity}
Let $1 < p <2$, $\frac{1}{p} + \frac{1}{p'}= 1$ and $\theta = \frac{2}{p'}$. Then
$$\norm{{\bf 1} \otimes \sum^n_{i=1}e_i \otimes e_i \otimes \delta_i}_{\K_{\Pi^o_1(OH_n, S^n_p)}} \sim \theta^{-1}(1-\theta)^{-\frac{3}{2}}n^{\frac{1}{p}}.$$
\end{thm}
\begin{proof}
First we consider the lower bound. Recall that the formal identity $$L_2(\nu) \otimes_{\pi} X \rightarrow L_2(\nu; X)$$
is a contraction for any measure $\nu$ and Banach spaces $X$ and $$L_2(f(t)dt) + L_2(g(t)dt) \cong L_2(\min\{f(t), g(t)\}dt)$$ isomorphically.
Then by Lemma \ref{lem-equi-identity} we have
\begin{align}
\lefteqn{\norm{{\bf 1} \otimes \sum^n_{i=1}e_i \otimes e_i \otimes \delta_i}^2_{\K_{\Pi^o_1(OH_n, S^n_p)}}} \nonumber \\ & \sim
\norm{{\bf 1}}^2_{L_2(n\mu_1) + L_2(n\mu_2) + L_2(n\mu_{3,1})\otimes_{\pi} L_2(n\mu_{3,2}) + \cdots + L_2(n\mu_{8,1})\otimes_{\pi} L_2(n\mu_{8,2})}
\label{FunctionSpace1} \\ 
& \geq \norm{{\bf 1}}^2_{L_2(n\mu_1) + L_2(n\mu_2) + L_2(n^2 \mu_{3,1}\times \mu_{3,2}) + \cdots + L_2(n^2\mu_{8,1}\times \mu_{8,2})}
\label{FunctionSpace2} \\
& \sim \int_{\mathbb{R}^4_+} \min(ns^{4-2\theta}t^{-2\theta}u^{-1}v^{-1}, ns^{-2\theta}t^{4-2\theta}uv,
n^2s^{4-2\theta}t^{-2\theta}u^{-1}v, n^2 s^{-2\theta}t^{4-2\theta}uv^{-1}, \nonumber \\ & \;\;\;\; n^2 s^{-2\theta}t^{-2\theta}u^{-1}v^{-1},
n^2 s^{-2\theta}t^{-2\theta}u^{-1}v, n^2 s^{-2\theta}t^{-2\theta}uv, n^2 s^{-2\theta}t^{-2\theta}uv^{-1} ) \frac{dsdtdudv}{stuv} \nonumber \\
& = \int_{\mathbb{R}^4_+} n^2s^{-1-2\theta}t^{-1-2\theta}\min(n^{-1}s^4u^{-2}v^{-2},
n^{-1}t^4, s^4u^{-2}, t^4v^{-2}, u^{-2}v^{-2}, u^{-2}, 1, v^{-2}) \nonumber \\ &\;\;\;\; dsdtdudv \nonumber \\ & = \int_{\mathbb{R}^4_+} G(s,t,u,v)\; dsdtdudv. \nonumber
\end{align}
Now we divide $\mathbb{R}^4_+$ into the regions according to the values of the minimum used in the integral above.
First we consider 8 regions $A_1, \cdots, A_8 \subseteq \mathbb{R}^4_+$ according to the values of
$\min(u^{-2}v^{-2}, u^{-2}, 1, v^{-2})$, and we further divide $A_i$'s ($1\leq i \leq 8$) into 3 sub-regions $A_{i,j}$ ($1\leq j\leq 3$)
according to the behavior of $s$ and $t$. See TABLE 1 in the next page for the details.
Note that if we take the transform $(s,t,u,v) \mapsto (t,s,u^{-1}, v^{-1})$ then the regions $A_5, \cdots, A_8$ and the associated integrand
correspond to those of $A_1, \cdots, A_4$, respectively, so that we are only to consider the cases $A_1, \cdots, A_4$.

\begin{table}
\caption{Regions}

\begin{center}\begin{minipage}{\textwidth}
\begin{tabular}{|c|c|c|l l|}

\hline \multirow{3}{0.5cm}{$A_1$} & \multirow{3}{4cm}{$0<u<1,\; 0< v < n^{-\frac{1}{2}}$} & $A_{1,1}$ & 
$s\geq u^{\frac{1}{2}},$& $t\geq n^{\frac{1}{4}}$ \\
\cline{3-5}
& & $A_{1,2}$ & $s < u^{\frac{1}{2}},$ & $ t\geq n^{\frac{1}{4}}u^{-\frac{1}{2}}s$ \\
\cline{3-5}
& & $A_{1,3}$ & $s\geq n^{-\frac{1}{4}}u^{\frac{1}{2}}t,$ & $ t < n^{\frac{1}{4}}$ \\

\hline \multirow{3}*{$A_2$} & \multirow{3}*{$0<u<1,\; n^{-\frac{1}{2}} \leq v < 1$} & $A_{2,1}$ &
$s\geq n^{\frac{1}{4}}u^{\frac{1}{2}}v^{\frac{1}{2}},$ & $ t\geq n^{\frac{1}{4}}$ \\
\cline{3-5}
& & $A_{2,2}$ & $s < n^{\frac{1}{4}}u^{\frac{1}{2}}v^{\frac{1}{2}},$ & $ t\geq u^{-\frac{1}{2}}u^{-\frac{1}{2}}s$ \\
\cline{3-5}
& & $A_{2,3}$ & $s\geq u^{\frac{1}{2}}u^{\frac{1}{2}}t,$ & $ t < n^{\frac{1}{4}}$ \\

\hline
\multirow{3}{0.5cm}{$A_3$} & \multirow{3}{4.5cm}{$0<u<1,\; n^{\frac{1}{2}} \leq v$} 
& $A_{3,1}$ & $s\geq n^{\frac{1}{4}}u^{\frac{1}{2}},$ & $ t\geq 1$\\
\cline{3-5}
& & $A_{3,2}$ & $s < n^{\frac{1}{4}}u^{\frac{1}{2}},$ & $ t\geq n^{-\frac{1}{4}}u^{-\frac{1}{2}}s$ \\
\cline{3-5}
& & $A_{3,3}$ & $s\geq n^{\frac{1}{4}}u^{\frac{1}{2}}t,$ & $ t < 1$ \\

\hline
\multirow{3}{0.5cm}{$A_4$} & \multirow{3}{4.5cm}{$0 < u < 1 \leq v < n^{\frac{1}{2}}$} 
& $A_{4,1}$ & $s\geq n^{\frac{1}{4}}u^{\frac{1}{2}},$ & $t\geq n^{\frac{1}{4}}v^{-\frac{1}{2}}$ \\
\cline{3-5}
& & $A_{4,2}$ & $s < n^{\frac{1}{4}}u^{\frac{1}{2}},$ & $ t\geq u^{-\frac{1}{2}}v^{-\frac{1}{2}}s$ \\
\cline{3-5}
& & $A_{4,3}$ & $s\geq u^{\frac{1}{2}}v^{\frac{1}{2}}t,$ & $ t < n^{\frac{1}{4}}v^{-\frac{1}{2}}$ \\

\hhline{|=====|} 
       $A_5$ \footnote{$A_{5,1}, \cdots, A_{8,3}$ are similarly determined but omitted.} 
       & $1\leq u ,\; n^{\frac{1}{2}} \leq v$   & $A_6$ & \multicolumn{2}{c|}{$1\leq u , \; 1 \leq v < n^{\frac{1}{2}}$}\\
\hline $A_7$ & $1\leq u ,\; 0 < v < n^{-\frac{1}{2}}$ & $A_8$ & \multicolumn{2}{c|}{$1\leq u , \; n^{-\frac{1}{2}} \leq v < 1$}\\
\hline

\end{tabular}\end{minipage}
\end{center}

\end{table}

The integrals over each regions are calculated in TABLE 2 in page 21. Note that the integrals over $A_{2,1}$ and $A_{4,1}$ are dominant with values 
$n^{1-\frac{\theta}{2}}\theta^{-1}$ when $\theta$ goes to $0$, and the integrals over $A_{2,3}$, $A_{4,2}$ and $A_{4,3}$ are dominant with values
$n^{1-\frac{\theta}{2}}(1-\theta)^{-\frac{3}{2}}$ when $\theta$ goes to $1$.
Thus, by combining all these calculations and $1-\frac{\theta}{2} = \frac{1}{p}$ we get the desired lower estimate $n^{\frac{1}{p}}\theta^{-1}(1-\theta)^{-\frac{3}{2}}$.

\begin{table} 
\caption{Integrals over the regions}
\begin{center}
\begin{tabular}{|c|c|c|c|}
\hline
\begin{minipage}[c]{1cm}
Region\\ \centering $A_{i,j}$
\end{minipage}
& 
\begin{minipage}[c]{3.5cm}
$$\Big(\int_{A_{i,j}} G\; dsdtdudv \Big)^{\frac{1}{2}}$$ : {\small The calculations below are only equivalent to the corresponding integral.}
\end{minipage}
&
\begin{minipage}[c]{2.5cm}
\centering Corresponding\\ \centering Function Space in \eqref{FunctionSpace2}
\end{minipage}
&
\begin{minipage}[c]{3cm}
\centering Corresponding\\ \centering Function Space in \eqref{FunctionSpace1}
\end{minipage}\\
\hline
$A_{1,1}$ &
\begin{minipage}[c]{3.5cm}
$$n^{\frac{3-\theta}{4}}\theta^{-1}(1-\theta)^{-\frac{1}{2}}$$
\end{minipage} & $L_2(n^2 \mu_{7,1}\times \mu_{7,2})$ & $L_2(n\mu_{7,1})\otimes_{\pi} L_2(n\mu_{7,2})$ \\
\hline
$A_{1,2}$ &
\begin{minipage}[c]{3.5cm}
$$n^{\frac{3-\theta}{4}}\theta^{-\frac{1}{2}}(1-\theta)^{-1}$$
\end{minipage} & $L_2(n^2\mu_{3,1}\times \mu_{3,2})$ & $L_2(n\mu_{3,1})\otimes_{\pi} L_2(n\mu_{3,2})$ \\
\hline
$A_{1,3}$ &
\begin{minipage}[c]{3.5cm}
$$n^{\frac{3-\theta}{4}}\theta^{-\frac{1}{2}}(1-\theta)^{-1}$$
\end{minipage} & $L_2(n\mu_2)$ & $L_2(n\mu_2)$ \\
\hline
$A_{2,1}$ &
\begin{minipage}[c]{3.5cm}
$$n^{1-\frac{\theta}{2}}\theta^{-1}(1-\theta)^{-1}$$
\end{minipage} & $L_2(n^2\mu_{7,1}\times \mu_{7,2})$ & $L_2(n\mu_{7,1})\otimes_{\pi} L_2(n\mu_{7,2})$ \\
\hline
$A_{2,2}$ &
\begin{minipage}[c]{3.5cm}
$$n^{1-\frac{\theta}{2}}\theta^{-\frac{1}{2}}(1-\theta)^{-1}$$
\end{minipage} & $L_2(n\mu_1)$ & $L_2(n\mu_1)$ \\
\hline
$A_{2,3}$ &
\begin{minipage}[c]{3.5cm}
$$n^{1-\frac{\theta}{2}}\theta^{-\frac{1}{2}}(1-\theta)^{-\frac{3}{2}}$$
\end{minipage} & $L_2(n\mu_2)$ & $L_2(n\mu_2)$ \\
\hline
$A_{3,1}$ &
\begin{minipage}[c]{3.5cm}
$$n^{\frac{3-\theta}{4}}\theta^{-1}(1-\theta)^{-\frac{1}{2}}$$
\end{minipage} & $L_2(n^2\mu_{8,1}\times \mu_{8,2})$ & $L_2(n\mu_{8,1})\otimes_{\pi} L_2(n\mu_{8,2})$ \\
\hline
$A_{3,2}$ &
\begin{minipage}[c]{3.5cm}
$$n^{\frac{3-\theta}{4}}\theta^{-\frac{1}{2}}(1-\theta)^{-1}$$
\end{minipage} & $L_2(n\mu_1)$ & $L_2(n\mu_1)$ \\
\hline
$A_{3,3}$ &
\begin{minipage}[c]{3.5cm}
$$n^{\frac{3-\theta}{4}}\theta^{-\frac{1}{2}}(1-\theta)^{-1}$$
\end{minipage} & $L_2(n^2\mu_{4,1}\times \mu_{4,2})$ & $L_2(n\mu_{4,1})\otimes_{\pi} L_2(n\mu_{4,2})$ \\
\hline
$A_{4,1}$ &
\begin{minipage}[c]{3.5cm}
$$n^{1-\frac{\theta}{2}}\theta^{-1}(1-\theta)^{-1}$$
\end{minipage} & $L_2(n^2\mu_{8,1}\times \mu_{8,2})$ & $L_2(n\mu_{8,1})\otimes_{\pi} L_2(n\mu_{8,2})$ \\
\hline
$A_{4,2}$ &
\begin{minipage}[c]{3.5cm}
$$n^{1-\frac{\theta}{2}}\theta^{-\frac{1}{2}}(1-\theta)^{-\frac{3}{2}}$$
\end{minipage} & $L_2(n\mu_1)$ & $L_2(n\mu_1)$ \\
\hline
$A_{4,3}$ &
\begin{minipage}[c]{3.5cm}
$$n^{1-\frac{\theta}{2}}\theta^{-\frac{1}{2}}(1-\theta)^{-\frac{3}{2}}$$
\end{minipage} & $L_2(n\mu_2)$ & $L_2(n\mu_2)$ \\
\hline

\end{tabular}
\end{center}
\end{table}

Now we consider the upper estimate. We use the same regions and fortunately that is enough. Indeed, we have
\begin{align*}
\lefteqn{\norm{{\bf 1} \otimes \sum^n_{i=1}e_i \otimes e_i \otimes \delta_i}_{\K_{\Pi^o_1(OH_n, S^n_p)}}}\\
& = \norm{({\bf 1}_{A_{1,1}} + \cdots + {\bf 1}_{A_{4,3}} + {\bf 1}_{A_{5,1}} + \cdots + 
{\bf 1}_{A_{8,3}} ) \otimes \sum^n_{i=1}e_i \otimes e_i \otimes \delta_i}_{\K_{\Pi^o_1(OH_n, S^n_p)}}\\
& \leq \norm{\sum_{A \in R_1} {\bf 1}_A }_{L_2(n\mu_1)} + \norm{\sum_{A \in R_2} {\bf 1}_A }_{L_2(n\mu_2)} + \cdots \\
& \;\;\;\; + \norm{\sum_{A \in R_3}{\bf 1}_A}_{L_2(n\mu_{3,1})\otimes_{\pi}L_2(n\mu_{3,2})} + \cdots
+ \norm{\sum_{A \in R_8}{\bf 1}_A}_{L_2(n\mu_{8,1})\otimes_{\pi}L_2(n\mu_{8,2})},
\end{align*}
where $$R_l := \{ A_{i,j} : A_{i,j}\; \text{corresponds to}\; L_2(n\mu_l) \; \text{in \eqref{FunctionSpace1}}\}$$ for $l =1,2$ and 
$$R_l := \{ A_{i,j} : A_{i,j}\; \text{corresponds to}\; L_2(n\mu_{l,1})\otimes_{\pi}L_2(n\mu_{l,2}) \; \text{in \eqref{FunctionSpace1}}\}$$
for $3\leq l \leq 8$. Thus, we get the upper bound of $\norm{{\bf 1} \otimes \sum^n_{i=1}e_i \otimes e_i \otimes \delta_i}_{\K_{\Pi^o_1(OH_n, S^n_p)}}$, namely the sum of norms of ${\bf 1}_{A_{i,j}}$'s calculated in the corresponding function spaces in \eqref{FunctionSpace1}.
However, this is the same as the lower bound which is nothing but the sum of norms of ${\bf 1}_{A_{i,j}}$'s calculated in the corresponding function spaces in \eqref{FunctionSpace2}.

Indeed, the terms corresponding to $L_2(\mu_1)$ or $L_2(\mu_2)$ are no problem since we calculate the norm in the same space.
For the remaining problematic terms we observe the following.
For example, if we consider the region $$A_{1,2} = \{0<u<1,\; s < u^{\frac{1}{2}},\;
t\geq n^{\frac{1}{4}}u^{-\frac{1}{2}}s \} \times \{0< v < n^{-\frac{1}{2}}\},$$ then we need to compare two norms calculated in $$L_2(\mu_{3,1})\otimes_{\pi}L_2(\mu_{3,2}) = L_2(s^{4-2\theta}t^{-2\theta}u^{-1}\frac{dsdtdu}{stu})\otimes_{\pi}L_2(v\frac{dv}{v})$$ and
$$L_2(\mu_{3,1}\times \mu_{3,2}) = L_2(s^{4-2\theta}t^{-2\theta}u^{-1}\frac{dsdtdu}{stu})\otimes_2 L_2(v\frac{dv}{v}),$$
which are the same since we have the separation of variables $(s,t,u)$ and $v$ and then the norms are just the product of two $L_2$-norms. 

Let's check another one. If we consider the region $$A_{4,1} = \{0 < u < 1,\; s\geq n^{\frac{1}{4}}u^{\frac{1}{2}}\} 
\times \{1 \leq v < n^{\frac{1}{2}}, \; t\geq n^{\frac{1}{4}}v^{-\frac{1}{2}}\},$$ then we need to compare two norms calculated in $$L_2(\mu_{8,1})\otimes_{\pi}L_2(\mu_{8,2}) = L_2(s^{-2\theta}u\frac{dsdu}{su})\otimes_{\pi}L_2(t^{-2\theta}v^{-1}\frac{dtdv}{tv})$$ and
$$L_2(\mu_{8,1}\times \mu_{8,2}) = L_2(s^{-2\theta}u\frac{dsdu}{su})\otimes_2 L_2(t^{-2\theta}v^{-1}\frac{dtdv}{tv}),$$
which are the same since we have the separation of variables $(s,u)$ and $(t,v)$ as we wanted. 

Similarly we can easily check that this separation of variables happens in every problematic terms, which leads us to the desired upper bound.
\end{proof}

\begin{rem}{\rm
When $\theta = 1$ we recover the well known $\sqrt{1 + \log n}\,$ factor (Proposition 4.9 of \cite{J05})
in the integral over every subregion of $A_2$, $A_4$, $A_6$ and $A_8$.
}
\end{rem}

\section{An application of Orlicz spaces}\label{sec-Orlicz}

In this section we will show that the result for the identity in the previous section is enough to conclude our final goal.
First we will look at the diagonal part to see that it is equivalent to an Orlicz sequence space,
and for the whole matrix we will consider its vector valued case. This ``Orlicz space argument" goes back to an unpublished result of Junge and Xu
and is also used by K. L. Yew in \cite{Y-p-summing}.

We consider the function $\Psi$ defined on $[0,\infty)$ by
\begin{align*}
\Psi(x) & = \inf_{{\bf 1} = f_1 + \cdots + f_8} x^2 \norm{f_1}^2_{L_2(\mu_1)} + x^2 \norm{f_2}^2_{L_2(\mu_2)}
+ x \norm{f_3}_{L_2(\mu_{3,1})\otimes_{\pi}L_2(\mu_{3,2})} + \cdots\\ & \;\;\;\;\;\;\;\;\;\;\;\;\;\;\;\;\;\;\;\;
+ x \norm{f_8}_{L_2(\mu_{8,1})\otimes_{\pi}L_2(\mu_{8,2})}.
\end{align*}

\begin{lem}\label{lem-equi-orlicz}
$\Psi$ is equivalent to a Orlicz function $\widetilde{\Psi}$.
\end{lem}
\begin{proof}
Clearly we have $\Psi(0) = 0$ and $\lim_{x\rightarrow \infty} \Psi(x) = \infty$.
Since we have
\begin{align*}
\frac{\Psi(x)}{x} & = \inf_{{\bf 1} = f_1 + \cdots + f_8} x \norm{f_1}^2_{L_2(\mu_1)} + x \norm{f_2}^2_{L_2(\mu_2)}
+ \norm{f_3}_{L_2(\mu_{3,1})\otimes_{\pi}L_2(\mu_{3,2})} + \cdots\\ & \;\;\;\;\;\;\;\;\;\;\;\;\;\;\;\;\;\;\;\;
+ \norm{f_8}_{L_2(\mu_{8,1})\otimes_{\pi}L_2(\mu_{8,2})}
\end{align*}
it is also clear that $\frac{\Psi(x)}{x}$ is an increasing function.

Now we consider the convex function $\widetilde{\Psi}(x) = \inf \{ f(x) : f\in \F_x\}$, where $\F_x$ is the set of
all linear functions intersecting at least two distinct points with the graph of $\Psi$.
Then by Lemma 1.e.7 of \cite{LTz} we have $$\frac{\Psi(x)}{4} \leq \Psi(\frac{x}{2}) \leq \widetilde{\Psi}(x) \leq \Psi(x).$$

\end{proof}

Due to the previous lemma we can consider the Orlicz sequence space $\ell_{\widetilde{\Psi}}$ defined by
	$$\ell_{\widetilde{\Psi}} = \{(a_n) : \sum_{n\geq 1}\widetilde{\Psi}\Big(\frac{\abs{a_n}}{\rho}\Big)
	< \infty \;\text{for some}\; \rho >0\}$$ and $$\norm{(a_n)}_{\widetilde{\Psi}} = \inf\{ \rho > 0 :
	\sum_{n\geq 1}\widetilde{\Psi}\Big(\frac{\abs{a_n}}{\rho}\Big)\leq 1\}.$$
We recover a similar form of our function space by a standard argument.
	\begin{lem}\label{lem-recover-form}
		\begin{align*}
		\norm{(a_n)}_{\widetilde{\Psi}} & \sim \inf\{ \norm{g_1}_{L_2(\mu_1 ;\ell_2)} + \norm{g_2}_{L_2(\mu_2; \ell_2)}
		+ \norm{g_3}_{L_2(\mu_{3,1})\otimes_{\pi}L_2(\mu_{3,2})\otimes_{\pi}\ell_1}\\ & \;\;\;\;\;\;\;\;\;\; + \cdots
		+ \norm{g_8}_{L_2(\mu_{8,1})\otimes_{\pi}L_2(\mu_{8,2})\otimes_{\pi}\ell_1}\},
		\end{align*}
	where the infimum runs over all possible $g_1 = (g^n_1)_{n}, \cdots, g_8 = (g^n_8)_{n}$ with $${\bf 1}\otimes a_n = g^n_1 + \cdots + g^n_8.$$
	\end{lem}
\begin{proof}
Let $R[(a_n)]$ be the right side. Suppose we have $\norm{(a_n)}_{\widetilde{\Psi}} < 1$, then, by Lemma \ref{lem-equi-orlicz} we can choose
	$${\bf 1} = f^n_1 + \cdots + f^n_8\,\; \text{for each $n$}$$
satisfying
	\begin{align*}
	& \sum_n \big[ \abs{a_n}^2 \norm{f^n_1}^2_{L_2(\mu_1)} + \abs{a_n}^2 \norm{f^n_2}^2_{L_2(\mu_2)}
	+ \abs{a_n} \norm{f^n_3}_{L_2(\mu_{3,1})\otimes_{\pi}L_2(\mu_{3,2})}\\
	& \;\;\;\;\;\;\;\; + \cdots + \abs{a_n} \norm{f^n_8}_{L_2(\mu_{8,1})\otimes_{\pi}L_2(\mu_{8,2})} \big]< 4.
	\end{align*}
Then, we have
	\begin{align*}
	\sum_n \abs{a_n}^2 \norm{f^n_1}^2_{L_2(\mu_1)} & ,\; \sum_n \abs{a_n}^2 \norm{f^n_2}^2_{L_2(\mu_2)},\;
	\sum_n \abs{a_n} \norm{f^n_3}_{L_2(\mu_{3,1})\otimes_{\pi}L_2(\mu_{3,2})},\\ \cdots \; & , \;
	\sum_n \abs{a_n} \norm{f^n_8}_{L_2(\mu_{8,1})\otimes_{\pi}L_2(\mu_{8,2})} < 4
	\end{align*}
which implies
	\begin{align*}
	R[(a_n)]& \leq \Big(\sum_n \abs{a_n}^2 \norm{f^n_1}^2_{L_2(\mu_1)}\Big)^{\frac{1}{2}} +
	\Big(\sum_n \abs{a_n}^2 \norm{f^n_2}^2_{L_2(\mu_2)}\Big)^{\frac{1}{2}}\\
	& \;\;\;\; + \sum_n \abs{a_n} \norm{f^n_3}_{L_2(\mu_{3,1})\otimes_{\pi}L_2(\mu_{3,2})}\\& \;\;\;\;
	+ \cdots + \sum_n \abs{a_n}\norm{f_8}_{L_2(\mu_{8,1})\otimes_{\pi}L_2(\mu_{8,2})} < 32
	\end{align*}
by setting $g^n_l = a_n\otimes f^n_l$ for $1\leq l \leq 8$ and $n\geq 1$.
Thus, we get
	$$R[(a_n)] \leq 32\norm{(a_n)}_{\widetilde{\Psi}}.$$

For the converse we assume that $R[(a_n)]< 1$. Then we can choose
$${\bf 1}\otimes a_n = g^n_1 + \cdots + g^n_8$$ such that
\begin{align*}
&\norm{g_1}_{L_2(\mu_1 ;\ell_2)} + \norm{g_2}_{L_2(\mu_2; \ell_2)}
+ \norm{g_3}_{L_2(\mu_{3,1})\otimes_{\pi}L_2(\mu_{3,2})\otimes_{\pi}\ell_1}\\ & \;\;\;\; + \cdots
+ \norm{g_8}_{L_2(\mu_{8,1})\otimes_{\pi}L_2(\mu_{8,2})\otimes_{\pi}\ell_1} < 1,
\end{align*}
which means
\begin{align*}
&\frac{\sum_n\norm{g^n_1}^2_{L_2(\mu_1)}}{8^2},\; \frac{\sum_n \norm{g^n_2}^2_{L_2(\mu_2)}}{8^2} ,\;
\frac{\sum_n \norm{g^n_3}_{L_2(\mu_{3,1})\otimes_{\pi}L_2(\mu_{3,2})}}{8}, \\ & \;\;\;\; \cdots
, \frac{\sum_n \norm{g^n_8}_{L_2(\mu_{8,1})\otimes_{\pi}L_2(\mu_{8,2})}}{8}  < \frac{1}{8}.
\end{align*}
Thus, by observing ${\bf 1} = a^{-1}_n g^n_1 + \cdots + a^{-1}_n g^n_8$ for non-zero $a_n$, we have
\begin{align*}
\sum_{n\geq 1}\widetilde{\Psi}\Big(\frac{\abs{a_n}}{8}\Big) &\leq \sum_{n\geq 1}\Psi\Big(\frac{\abs{a_n}}{8}\Big)\\
& \leq \sum_{n\geq 1}\Big( \frac{\norm{g^n_1}^2_{L_2(\mu_1)}}{8^2} + \frac{\norm{g^n_2}^2_{L_2(\mu_2)}}{8^2} +
\frac{\norm{g^n_3}_{L_2(\mu_{3,1})\otimes_{\pi}L_2(\mu_{3,2})}}{8}\\ &\;\;\;\;\;\;\;\;\;\;\;\; + \cdots + \frac{\norm{g^n_8}_{L_2(\mu_{8,1})\otimes_{\pi}L_2(\mu_{8,2})}}{8} \Big) < 1,
\end{align*}
which means
	$$\norm{(a_n)}_{\ell_{\widetilde{\Psi}}} < 8.$$

\end{proof}

In the case of identity we can further simplify the calculation by the averaging trick.

	\begin{lem}\label{lem-recover-form-identity}
		\begin{align*}
		\norm{\sum^n_{i=1}e_i}_{\widetilde{\Psi}} & \sim \inf_{{\bf 1} = f_1 + \cdots + f_8} n^{\frac{1}{2}} \norm{f_1}_{L_2(\mu_1)}
		+ n^{\frac{1}{2}}\norm{f_2}_{L_2(\mu_2)}\\ & \;\;\;\;\;\;\;\;\;\;\;\;\;\;\;\;\;\;\;\; 
		+ n \norm{f_3}_{L_2(\mu_{3,1})\otimes_{\pi}L_2(\mu_{3,2})} + \cdots + n \norm{f_8}_{L_2(\mu_{8,1})\otimes_{\pi}L_2(\mu_{8,2})}.
		\end{align*}
	\end{lem}
\begin{proof}
Let
	\begin{align*}
	A & := \norm{g_1}_{L_2(\mu_1 ;\ell_2)} + \norm{g_2}_{L_2(\mu_2; \ell_2)}
	+ \norm{g_3}_{L_2(\mu_{3,1})\otimes_{\pi}L_2(\mu_{3,2})\otimes_{\pi}\ell_1}\\ & \;\;\;\; + \cdots
	+ \norm{g_8}_{L_2(\mu_{8,1})\otimes_{\pi}L_2(\mu_{8,2})\otimes_{\pi}\ell_1}
	\end{align*}
for fixed $g_1 = (g^i_1)^n_{i=1}, \cdots, g_8 = (g^i_8)^n_{i=1}$ with $1 = g^i_1 + \cdots + g^i_8$.
Now we set
	$$f_l = \frac{1}{\abs{S_n}}\sum_{\sigma \in S_n}g^{\sigma(i)}_l,\; 1\leq l \leq 8,$$
where $S_n$ is the permutation group of $\{1,\cdots ,n\}$. Then for $l = 1,2$ we have
	\begin{align*}
	n^{\frac{1}{2}}\norm{f_l}_{L_2(\mu_l)} & = \norm{\sum^n_{i=1}f_l \otimes e_i}_{L_2(\mu_l;\ell_2)}\\
	& \leq \frac{1}{\abs{S_n}}\sum_{\sigma \in S_n}\norm{\sum^n_{i=1}g^{\sigma(i)}_l\otimes e_i}_{L_2(\mu_l;\ell_2)}\leq \norm{g_l}_{L_2(\mu_l ;\ell_2)}.
	\end{align*}
Similarly, we have
	$$n\norm{f_l}_{L_2(\mu_{l,1})\otimes_{\pi}L_2(\mu_{l,2})} \leq \norm{g_l}_{L_2(\mu_{l,1})\otimes_{\pi}L_2(\mu_{l,2})\otimes_{\pi}\ell_1}$$
for $3\leq l \leq 8$.

Consequently, we have
	\begin{align*}
	A & \geq n^{\frac{1}{2}} \norm{f_1}_{L_2(\mu_1)} + n^{\frac{1}{2}}\norm{f_2}_{L_2(\mu_2)}
	+ n \norm{f_3}_{L_2(\mu_{3,1})\otimes_{\pi}L_2(\mu_{3,2})}\\ &\;\;\;\; + \cdots
	+ n \norm{f_8}_{L_2(\mu_{8,1})\otimes_{\pi}L_2(\mu_{8,2})},
	\end{align*}
which leads us to the desired conclusion by Lemma \ref{lem-recover-form}.

\end{proof}

Now we prove Lemma \ref{lem-equi-identity}.

{\it (proof of Lemma \ref{lem-equi-identity})}
	$$\norm{{\bf 1} \otimes \sum^n_{i=1}e_i \otimes e_i \otimes \delta_i}_{\K_{\Pi^o_1(OH_n, S^n_p)}}
	= \inf \Big\{ \sum^8_{l=1}\norm{h_l}_{F_l}\Big\},$$
where the infimum runs over all possible decomposition
	$${\bf 1}\otimes \sum^n_{i=1}e_i \otimes e_i \otimes \delta_i = h_1 + \cdots + h_8.$$
For a given $\epsilon > 0$ we consider a decomposition $(h_l)^8_{l=1}$ with
	$$\sum^8_{l=1}\norm{h_l}_{F_l} \leq (1+\epsilon)\norm{{\bf 1} \otimes \sum^n_{i=1}e_i \otimes e_i \otimes \delta_i}_{\K_{\Pi^o_1(OH_n, S^n_p)}},$$
and let
	$$h_l = \sum^n_{i,j,k =1} h_l^{(i,j,k)} \otimes e_i \otimes e_j \otimes e_k$$
with scalar-valued $h_l^{(i,j,k)}$ for $1\leq l \leq 8$.

If we consider the diagonal projection
	$$P : \ell^n_2 \otimes \ell^n_2 \otimes \ell^n_2 \rightarrow \ell^n_2 \otimes \ell^n_2 \otimes \ell^n_2,
	\; e_i \otimes e_j \otimes e_k \mapsto \delta_{i,j,k}e_i \otimes e_i \otimes e_k,$$
then we have
	\begin{align*}
	{\bf 1}\otimes \sum^n_{i=1}e_i \otimes e_i \otimes \delta_i & = (I\otimes P)
	\Big({\bf 1}\otimes \sum^n_{i=1}e_i \otimes e_i \otimes \delta_i \Big)\\
	& = (I\otimes P)\sum^8_{l=1}h_l = \sum^8_{l=1}\sum^n_{i = 1}h_l^{(i,i,i)} \otimes e_i \otimes e_i \otimes \delta_i
	\end{align*}
and
	$$\sum^8_{l=1}\norm{(I\otimes P)h_l}_{F_l}\leq \sum^8_{l=1}\norm{h_l}_{F_l} \leq
	(1+\epsilon)\norm{{\bf 1} \otimes \sum^n_{i=1}e_i \otimes e_i \otimes \delta_i}_{\K_{\Pi^o_1(OH_n, S^n_p)}}.$$
Indeed, we are only to check that $P$ is completely contractive as mappings on
$C_n \widehat{\otimes}\, C_n \widehat{\otimes}\, C_n$, $R_n \widehat{\otimes}\, R_n \widehat{\otimes}\, R_n$,
$C_n \widehat{\otimes}\, C_n \widehat{\otimes}\, R_n$ and $R_n \widehat{\otimes}\, R_n \widehat{\otimes}\, C_n$.
The first two cases are clear since column and row Hilbert spaces are homogeneous,
i.e. every bounded maps are completely bounded with the same cb-norm.

For $P : C_n \widehat{\otimes}\, C_n \widehat{\otimes}\, R_n \rightarrow C_n \widehat{\otimes}\, C_n \widehat{\otimes}\, R_n$
we consider the factorization
	$$P : C_n \widehat{\otimes}\, C_n \widehat{\otimes}\, R_n \stackrel{Q\otimes I_{R_n}}{\longrightarrow}
	C_n \widehat{\otimes}\, C_n \widehat{\otimes}\, R_n \stackrel{I_{C_n}\otimes Q}{\longrightarrow}
	C_n \widehat{\otimes}\, C_n \widehat{\otimes}\, R_n,$$ where $$Q : \ell^n_2 \otimes \ell^n_2 \rightarrow \ell^n_2 \otimes \ell^n_2,
	\; e_i \otimes e_j \mapsto \delta_{i,j}e_i \otimes e_i.$$
Since $Q$ is completely contractive as mappings on $C_n \widehat{\otimes}\, R_n$
and $C_n \widehat{\otimes}\, R_n$ we get the desired conclusion. The last case is obtained similarly.

By looking at the coefficient of $e_i \otimes e_i \otimes \delta_i$ we observe that
	$$\sum^8_{l=1}h_l^{(i)} = {\bf 1}$$
for all $1\leq i\leq n$, where $h_l^{(i)} = h_l^{(i,i,i)}$. If we set
	$$\rho = \norm{{\bf 1} \otimes \sum^n_{i=1}e_i \otimes e_i \otimes \delta_i}_{\K_{\Pi^o_1(OH_n, S^n_p)}},$$
then we have
	\begin{align*}
	\sum^n_{i=1}\widetilde{\Psi}\Big(\frac{1}{8\rho}\Big) &\leq \sum^n_{i=1}\Psi\Big(\frac{1}{8\rho}\Big)\\
	& \leq \sum^n_{i=1}\Big( \frac{1}{64\rho^2} \norm{h^{(i)}_1}^2_{L_2(\mu_1)} + \frac{1}{64\rho^2} \norm{h^{(i)}_2}^2_{L_2(\mu_2)}\\
	& \;\;\;\; + \frac{1}{8\rho} \norm{h^{(i)}_3}_{L_2(\mu_{3,1})\otimes_{\pi}L_2(\mu_{3,2})} + \cdots
	+ \frac{1}{8\rho} \norm{h^{(i)}_8}_{L_2(\mu_{8,1})\otimes_{\pi}L_2(\mu_{8,2})}\Big)\\ & \leq 1 + \epsilon,
	\end{align*}
since we have $$\norm{(1\otimes P)h_l}^2_{F_l} = \sum^n_{i=1} \norm{h^{(i)}_l}^2_{L_2(\mu_l)}$$
for $l=1,2$ and $$\norm{(1\otimes P)h_l}_{F_l} = \sum^n_{i=1} \norm{h^{(i)}_l}_{L_2(\mu_{l,1})\otimes_{\pi}L_2(\mu_{l,2})}$$
for $3 \leq l \leq 8$.

Thus, by Lemma \ref{lem-recover-form-identity} we have
	\begin{align*}
	&\inf_{{\bf 1} = f_1 + \cdots + f_8} n^{\frac{1}{2}} \norm{f_1}_{L_2(\mu_1)} + n^{\frac{1}{2}} \norm{f_2}_{L_2(\mu_2)}
	+ n \norm{f_3}_{L_2(\mu_{3,1})\otimes_{\pi}L_2(\mu_{3,2})} \\ & \;\;\;\;\;\;\;\;\;\;\;\;\;\;\;\;
	+ \cdots + n \norm{f_8}_{L_2(\mu_{8,1})\otimes_{\pi}L_2(\mu_{8,2})}\\ & \sim \norm{\sum^n_{i=1}e_i}_{\widetilde{\Psi}}
	\leq 8\norm{{\bf 1} \otimes \sum^n_{i=1}e_i \otimes e_i \otimes \delta_i}_{\K_{\Pi^o_1(OH_n, S^n_p)}}.
	\end{align*}
The converse inequality is clear.

\vspace{1cm}

\begin{prop}\label{prop-orlicz-lp}
Let $1 < p <2$, $\frac{1}{p} + \frac{1}{p'}= 1$ and $\theta =
\frac{2}{p'}$. Then we have the inclusion $\ell_p \subseteq \ell_{\widetilde{\Psi}}$ with norm $\lesssim \theta^{-1}(1-\theta)^{-\frac{3}{2}}$.
\end{prop}
\begin{proof}
Note that $\ell_p$ and $\ell_{\widetilde{\Psi}}$ are both Orlicz sequence spaces. Thus, by Proposition 4.a.5. in \cite{LTz} it is enough to check that
if there is a constant $C>0$ such that $$\norm{\sum^n_{i=1}e_i}_{\widetilde{\Psi}}
\leq C \norm{\sum^n_{i=1}e_i}_{\ell_p} = C \theta^{-1}(1-\theta)^{-\frac{3}{2}} n^{\frac{1}{p}}$$ for any $n\in \n$, which is assured by Theorem \ref{thm-calculation-identity}.
\end{proof}

Finally we prove our main result.
\begin{thm}
Let $1< p < 2$ and $\frac{1}{p} + \frac{1}{p'} = 1$. Then, for any Hilbert space $H$ we have 
	$$CB(B(H), OH) \subseteq \Pi_{p', cb}(B(H), OH)$$ with the norm $\lesssim \big( \frac{p'}{p'-2} \big)^{\frac{1}{2}}$.
Equivalently, we have 
	$$\pi^o_1(T_x : OH \rightarrow \ell_p) \lesssim \Big( \frac{p'}{p'-2} \Big)^{\frac{1}{2}} \norm{x}_{\ell_p(OH)}$$
for all $x \in \ell_p(OH)$ and $T_x : OH \rightarrow \ell_{p}$, the linear map naturally associated to $x$.
\end{thm}
\begin{proof}
We focus on the $n$-dimensional case as before. Let
	$$a = \sum^n_{i,j=1}a_{ij} e_j \otimes e_{ii} \in OH_n\otimes S^n_p.$$
Suppose
	$$\norm{(a_{ij})}_{\ell_{\widetilde{\Psi}}(\ell^n_2)} 	
	=\norm{\Big(\Big[\sum^n_{j=1}\abs{a_{ij}}^2\Big]^{\frac{1}{2}}\Big)^n_{i=1}}_{\ell_{\widetilde{\Psi}}} < 1.$$
Then there are $g_1 = (g^i_1)^n_{i=1}, \cdots, g_8 = (g^i_8)^n_{i=8}$ with
	$${\bf 1} = g^i_1 + \cdots + g^i_8$$
such that
	\begin{align*}
	4 > & \sum^n_{i=1}\Big[ \sum^n_{j=1}\abs{a_{ij}}^2\norm{g^i_1}^2_{L_2(\mu_1)} + \sum^n_{j=1}\abs{a_{ij}}^2\norm{g^i_2}^2_{L_2(\mu_2)}\\
	& \;\;\;\;\;\;\;\;+ \Big(\sum^n_{j=1}\abs{a_{ij}}^2\Big)^{\frac{1}{2}}\norm{g^i_3}_{L_2(\mu_{3,1})\otimes_{\pi}L_2(\mu_{3,2})}\\
	& \;\;\;\;\;\;\;\;+ \cdots + \Big(\sum^n_{j=1}\abs{a_{ij}}^2\Big)^{\frac{1}{2}}\norm{g^i_8}_{L_2(\mu_{8,1})\otimes_{\pi}L_2(\mu_{8,2})}\Big].
	\end{align*}
If we set $f^{ij}_l = g^i_l \otimes a_{ij}$ for $1\leq l \leq 8$, then we have
	\begin{align*}
	4 & > \sum^n_{i=1}\Big[ \norm{(f^{ij}_1)^n_{j=1}}^2_{L_2(\mu_1; \ell^n_2)} + \norm{(f^{ij}_2)^n_{j=1}}^2_{L_2(\mu_2; \ell^n_2)}
	\\& \;\;\;\;+ \norm{(f^{ij}_3)^n_{j=1}}_{L_2(\mu_{3,1})\otimes_{\pi}L_2(\mu_{3,2})\otimes_{\pi}\ell^n_2} + \cdots  +
	\norm{(f^{ij}_8)^n_{j=1}}_{L_2(\mu_{8,1})\otimes_{\pi}L_2(\mu_{8,2})\otimes_{\pi}\ell^n_2}\Big] \\
	& = \norm{(f^{ij}_1)^n_{i,j=1}}^2_{L_2(\mu_1; \ell^n_2(\ell^n_2))} + \norm{(f^{ij}_2)^n_{i,j=1}}^2_{L_2(\mu_2; \ell^n_2(\ell^n_2))}
	\\&\;\;\;\; + \norm{(f^{ij}_3)^n_{i,j=1}}_{L_2(\mu_{3,1})\otimes_{\pi}L_2(\mu_{3,2})\otimes_{\pi}\ell^n_1(\ell^n_2)}\\
	& \;\;\;\; + \cdots  + \norm{(f^{ij}_8)^n_{i,j=1}}_{L_2(\mu_{8,1})\otimes_{\pi}L_2(\mu_{8,2})\otimes_{\pi}\ell^n_1(\ell^n_2)}.
	\end{align*}

Now we have by Corollary \ref{cor-SptimesOH-Embedding} that
	$$\pi^o_1(T_a) \sim \theta(1-\theta) \norm{{\bf 1}\otimes a}_{\K_{\Pi^o_1(OH_n, \ell^n_p)}}
	\leq \theta(1-\theta) \inf \sum^8_{l=1} \norm{f_l}_{F_l},$$
where $\theta = \frac{2}{p'}$ and the infimum above runs over all possible
	$${\bf 1}\otimes a = f_1 + \cdots + f_8.$$
Note that the formal identities
	$$L_2(\mu)\otimes_{\pi} X \rightarrow L_2(\mu ; X) \;\text{and}\; \ell^n_1(\ell^n_2) 
	= \ell^n_1 \otimes_{\pi} \ell^n_2 \rightarrow \ell^n_2\otimes_{\pi}\ell^n_2$$
are contractions for any Banach space $X$. Then, we have
	\begin{align*}
	\sum^8_{l=1} \norm{f_l}_{F_l} & \leq \norm{f_1}_{L_2(\mu_1; \ell^n_2(\ell^n_2))} + \norm{f_2}_{L_2(\mu_2; \ell^n_2(\ell^n_2))}
	\\ & \;\;\;\;+  \norm{f_3}_{L_2(\mu_{3,1})\otimes_{\pi}L_2(\mu_{3,2})\otimes_{\pi}\ell^n_2\otimes_{\pi}\ell^n_2}
	+ \norm{f_4}_{L_2(\mu_{4,1})\otimes_{\pi}L_2(\mu_{4,2})\otimes_{\pi}\ell^n_2\otimes_{\pi}\ell^n_2}\\
	& \;\;\;\; + \norm{f_5}_{L_2(\mu_{5,1})\otimes_{\pi}L_2(\mu_{5,2})\otimes_{\pi}\ell^n_1(\ell^n_2)} + \cdots
	+ \norm{f_8}_{L_2(\mu_{8,1})\otimes_{\pi}L_2(\mu_{8,2})\otimes_{\pi}\ell^n_1(\ell^n_2)}\\
	&\leq \norm{f_1}_{L_2(\mu_1; \ell^n_2(\ell^n_2))} + \norm{f_2}_{L_2(\mu_2; \ell^n_2(\ell^n_2))}
	\\ & \;\;\;\;+  \norm{f_3}_{L_2(\mu_{3,1})\otimes_{\pi}L_2(\mu_{3,2})\otimes_{\pi}\ell^n_1(\ell^n_2)}
	+ \norm{f_4}_{L_2(\mu_{4,1})\otimes_{\pi}L_2(\mu_{4,2})\otimes_{\pi}\ell^n_1(\ell^n_2)}\\
	& \;\;\;\; + \norm{f_5}_{L_2(\mu_{5,1})\otimes_{\pi}L_2(\mu_{5,2})\otimes_{\pi}\ell^n_1(\ell^n_2)} + \cdots
	+ \norm{f_8}_{L_2(\mu_{8,1})\otimes_{\pi}L_2(\mu_{8,2})\otimes_{\pi}\ell^n_1(\ell^n_2)}.
	\end{align*}
If we set $f_l = (f^{ij}_l)^n_{i,j=1}$, then we have
	$$\sum^8_{l=1} \norm{f_l}_{F_l} < 28.$$
Thus, we have
	$$\pi^o_1(T_a) \lesssim \theta(1-\theta)\norm{(a_{ij})^n_{i,j=1}}_{\ell_{\widetilde{\Psi}}(\ell^n_2)}.$$
Finally, by Proposition \ref{prop-orlicz-lp} we have
	$$\pi^o_1(T_a) \lesssim (1-\theta)^{-\frac{1}{2}}\norm{\sum^n_{i,j=1}a_{ij}e_{ii} \otimes e_j}_{\ell_p(\ell^n_2)}.$$

\end{proof}

\begin{rem}{\rm
A similar argument as above can be used to prove $(3'')$ of Remark \ref{rem-dualprob}. Let's describe it briefly.
Let $1< p <2$ and $\theta = \frac{1}{p}$.
First, we consider the embedding of
$$ C_p \hookrightarrow L^c_2(t^{-\theta}; \ell_2) +_2 L^r_2(t^{1-\theta}; \ell_2),\; e_i \mapsto
{\bf 1}\otimes e_i.$$ By a similar argument as in section \ref{subsec-L2L1}
it is well known that $$L^c_2(t^{-\theta}; \ell_2) +_2 L^r_2(t^{1-\theta}; \ell_2)$$ is completely complemented
in the predual of a von Nemann algebra with $QWEP$. For $OH$ we use the same embedding as before. Then, we have
$$\pi^o_1(T_x : OH \rightarrow C_p) \sim \norm{{\bf 1}\otimes x}_{\K_{\Pi^o_1(OH, C_p)}},$$ where
\begin{align*}
\begin{split}
\K_{\Pi^o_1(OH, C_p)} & = (L^c_2(t^{-\frac{1}{2}}; \ell_2) + L^r_2(t^{\frac{1}{2}}; \ell_2))
\widehat{\otimes} (L^c_2(s^{-\theta}; \ell_2) + L^r_2(s^{1-\theta}; \ell_2)) \\
& = L^c_2(t^{-\frac{1}{2}}s^{-\theta}; \ell_2\otimes \ell_2) + L^r_2(t^{\frac{1}{2}}s^{1-\theta}; \ell_2\otimes \ell_2)\\
& \;\;\;\; + L^c_2(t^{-\frac{1}{2}}; \ell_2)\widehat{\otimes} L^r_2(s^{1-\theta}; \ell_2) + L^r_2(t^{\frac{1}{2}}; \ell_2)
\widehat{\otimes} L^r_2(s^{-\theta}; \ell_2).
\end{split}
\end{align*}
When $x = \sum^n_{i=1}e_i\otimes e_i$ we can calculate $$\norm{{\bf 1}\otimes x}_{\K_{\Pi^o_1(OH, C_p)}} \sim (1-\theta)^{-\frac{1}{2}}(2\theta-1)^{-\frac{1}{2}}n^{\frac{p+2}{4p}}$$ as before.
(We divide $\mathbb{R}^2_+$ into four regions according to the minimum, then we get the lower bound
and the upper bound is the same since we have separation of variables for all problematic terms.)

Since we have $S_p(OH) = C_p \otimes_h OH \otimes_h R_p$ under the mapping
$$e_{ij} \otimes e_k \mapsto e_{i1}\otimes e_k \otimes e_{1j}$$
we are only to compare $\norm{{\bf 1}\otimes x}_{\K_{\Pi^o_1(OH, C_p)}}$ and $\norm{x}_{C_p \otimes_h OH}$.
Note that for any unitaries $U$ and $V$ we have
$$\norm{{\bf 1}\otimes UxV}_{\K_{\Pi^o_1(OH, C_p)}} = \norm{{\bf 1}\otimes x}_{\K_{\Pi^o_1(OH, C_p)}}$$ and
$$\norm{UxV}_{C_p \otimes_h OH} = \norm{x}_{C_p \otimes_h OH},$$ since
\begin{align*}
\begin{split}
C_p \otimes_h OH & = [C, R]_{\frac{1}{p}} \otimes_h OH = [C\otimes_h OH, R\otimes_h OH]_{\frac{1}{p}}\\
& = \Big[[C\otimes_h C, C\otimes_h R]_{\frac{1}{2}}, [R\otimes_h C, R\otimes_h R]_{\frac{1}{2}}\Big]_{\frac{1}{p}}\cong S_r
\end{split}
\end{align*}
isometrically for $r = \frac{4p}{p+2}$.

Thus it is enough to consider the case when $x$ is a diagonal matrix. Since the closed linear span of ${\bf 1}\otimes x$ and
$x$ for diagonal $x$ in $\K_{\Pi^o_1(OH, C_p)}$ and $C_p \otimes_h OH$, respectively, are equivalent to Orlicz sequence spaces
we are only to compare norms $\norm{{\bf 1}\otimes x}_{\K_{\Pi^o_1(OH, C_p)}}$ and $\norm{x}_{C_p \otimes_h OH}$ for
$x = \sum^n_{i=1}e_i\otimes e_i$, which is already done above.
}
\end{rem}

\bibliographystyle{amsplain}
\providecommand{\bysame}{\leavevmode\hbox
to3em{\hrulefill}\thinspace}

\end{document}